\documentclass[11pt]{article}
\usepackage[english,activeacute]{babel}
\usepackage{amsmath,amsfonts,amssymb,amstext,amsthm,amscd,mathrsfs,amsbsy,xypic,centernot}

\newtheorem{teo}{Theorem}[section]
\newtheorem{pro}[teo]{Proposition}
\newtheorem{coro}[teo]{Corollary}
\newtheorem{lem}[teo]{Lemma}
\newtheorem{conj}[teo]{Conjecture}

\theoremstyle{definition}
\newtheorem{defi}[teo]{Definition}
\newtheorem{exam}[teo]{Example}
\newtheorem{rem}[teo]{Remark}
\newtheorem{nota}[teo]{Notation}
\newtheorem{algo}[teo]{Algorithm}

\newcommand{\N}{\mathbb N}
\newcommand{\Z}{\mathbb Z}
\newcommand{\Q}{\mathbb Q}
\newcommand{\R}{\mathbb R}
\newcommand{\C}{\mathbb C}
\newcommand{\K}{\mathbb K}
\newcommand{\Pro}{\mathbb P^{\binom{\ls}{\ld}-1}}
\newcommand{\V}{\mathbf{V}}

\newcommand{\Ky}{\K[Y_1,\ldots,Y_s]}

\newcommand{\Kx}{\K[X_1,\ldots,X_d]}

\newcommand{\m}{\mathfrak{m}}
\newcommand{\n}{\mathfrak{n}}
\newcommand{\bn}{\bar{\n}}
\newcommand{\bm}{\bar{\m}}
\newcommand{\I}{\mathcal{I}}

\newcommand{\Ltn}{\Lambda_{t,n}}

\newcommand{\Ls}{\Lambda_{s,n}}
\newcommand{\Ld}{\Lambda_{d,n}}

\newcommand{\ls}{\lambda_{s,n}}
\newcommand{\ld}{\lambda_{d,n}}

\newcommand{\Xg}{X_{\Gamma}}

\newcommand{\vp}{\varphi}
\newcommand{\vpg}{\varphi_{A}}

\newcommand{\vps}{\varphi^{*}}
\newcommand{\vpb}{\bar{\vps}}
\newcommand{\vpd}{(\vpb)^{\vee}}

\newcommand{\Tx}{T_x^nX}
\newcommand{\Tpg}{T_p^nX_{\Gamma}}

\newcommand{\Dn}{D^n_x(\varphi)}
\newcommand{\Dng}{D^n_x(\varphi_{A})}
\newcommand{\Jac}{Jac_n(F)}

\newcommand{\Na}{Nash_n(X)}
\newcommand{\Nag}{Nash_n(X_{\Gamma})}

\newcommand{\Oxx}{\mathcal{O}_{X,x}}

\newcommand{\con}{\sigma}
\newcommand{\cond}{\sigma^{\vee}}
\newcommand{\NI}{\mathcal{N}(I)}

\newcommand{\Nas}{Nash_n(\Gamma)}

\newcommand{\Gan}{\Gamma^{(n)}}
\newcommand{\Ga}{\Gamma}
\newcommand{\be}{\beta}
\newcommand{\ga}{\gamma}
\newcommand{\al}{\alpha}

\newcommand*\quot[2]{{^{\textstyle #1}\big/_{\textstyle #2}}}

\begin{document}

\title{A higher-order tangent map and a conjecture on the higher Nash blowup of curves}

\author{Enrique Ch\'avez-Mart\'inez, Daniel Duarte\footnote{Research supported by CONACyT grant 287622.}, 
Arturo Giles Flores\footnote{Research supported by CONACyT grant 221635.}}

\maketitle

%\begin{titlepage}
%\begin{center}
%{\huge A higher-order tangent map and a conjecture on the higher Nash blowup of curves \par}
%\end{center}

%\begin{center}
%{\Large by Enrique Ch\'avez-Mart\'inez, \\ Daniel Duarte \\ and \\ Arturo Giles Flores\par}
%\end{center}
%\vspace{2cm}
%E. Ch\'avez-Mart\'inez\\
%Universidad Nacional Aut\'onoma de M\'exico.\\
%Address: Av. Universidad s/n. Col. Lomas de Chamilpa, C.P. 62210, Cuernavaca, Morelos, Mexico.\\
%Email: enrique.chavez@im.unam.mx
%\\
%\\
%D. Duarte (corresponding author)\footnote{Research supported by CONACyT grant 287622.}\\
%CONACyT-Universidad Aut\'onoma de Zacatecas.\\
%Address: Jard\'in Ju\'arez 147, Centro Hist\'orico, C.P. 98000, Zacatecas, Zacatecas, Mexico.\\
%Email: aduarte@uaz.edu.mx\\
%Telephone: (+52) 656 166 1230.
%\\
%\\
%A. Giles Flores\footnote{Research supported by CONACyT grant 221635.}\\
%Universidad Aut\'onoma de Aguascalientes.\\
%Address: Av. Universidad 940, Ciudad Universitaria, C. P. 20131, Aguascalientes, Ags., Mexico.\\
%Email: arturo.giles@cimat.mx
%\vfill
%\centering
%{\large \today\par}

%\end{titlepage}

\begin{abstract}
We introduce a higher-order version of the tangent map of a morphism and find a matrix representation. We then apply this
matrix to solve a conjecture by T. Yasuda regarding the semigroup of the higher Nash blowup of formal curves. We first show 
that the conjecture is true for toric curves. We conclude by exhibiting a family of non-monomial curves where the conjecture fails.
\end{abstract}

%\noindent\keywords{Higher-order tangent map, Nash blowups, Jacobian matrix, curves.}
%\\
%\msc{14B05,14M25,32S45.}

\section*{Introduction} 

In recent years several authors have introduced higher-order versions of the Jacobian matrix. In \cite{D3}, a higher-order 
Jacobian matrix is studied in relation with the higher Nash blowup of a hypersurface. More recently,
in \cite{BD,BJNB}, a similar matrix is introduced for any finitely generated algebra. In these articles the matrices are used to study 
singularities in arbitrary characteristic or to study algebraic properties of the module of K\"ahler differentials of high order. In another 
but related direction, in \cite{dFDo} it is described a matrix associated to a relative compactification of the induced map on the main
components of jet schemes of a projective birational morphism. 

In this paper we introduce a matrix that represents a higher-order tangent map of a morphism.  This matrix involves higher-order 
derivatives, making it more suitable for some computations related to jet-spaces. Our main application of this matrix is the solution 
of a conjecture by T. Yasuda related to the higher Nash blowup of formal curves.

The Nash blowup and the higher Nash blowup of an algebraic variety are modifications that replace singular points by limits of certain 
vector spaces associated to the variety at non-singular points. There are several questions relating these modifications to resolution of 
singularities (\cite{No, S, Y1}). Those questions have been treated in \cite{No,R,GS-1,GS-2,Hi,Sp,Y1,GT,GM, At, D1, DG, T}. 
In \cite{Y1}, T. Yasuda proved that the higher Nash blowup solves singularities of curves. In a subsequent paper, Yasuda gave a 
conjectural explicit description of the semigroup of the higher Nash blowup of formal curves. In this paper we show that the conjecture 
is true for toric curves but false in general.

To study Yasuda's conjecture in the case of toric curves we first develop a combinatorial description for the higher Nash blowup of 
toric varieties. This result is based and inspired in the analogous description of the usual Nash blowup of toric varieties given in \cite{GT, GM}. 
Our results depend strongly on the general framework developed in \cite{GT} for not necessarily normal toric varieties. With this general
combinatorial description at hand, we are able to prove Yasuda's conjecture for toric curves.

Finally, we present a family of non-monomial curves showing that Yasuda's conjecture fails in general. By combining the results
we obtained for monomial morphisms and the general construction of the matrix representing the higher-order tangent map, we are
able to describe a particular element of the semigroup of the higher Nash blowup of this family of curves which does not belong to
the semigroup suggested by Yasuda.

The present paper is organized as follows. In section 1 we introduce a higher-order tangent map, find a matrix representation and study
its basic properties. In section 2 we study in detail the higher-order Jacobian matrix of monomial morphisms. In section 3 we construct
a special cover for the higher Nash blowup of a toric variety. Section 4 is devoted to prove Yasuda's conjecture concerning
the explicit description of the semigroup of the higher Nash blowup of a toric curve. Finally, in section 5 we exhibit a family of 
non-monomial curves for which Yasuda's conjecture fails.
\\
\\
\noindent\textbf{Convention: }Throughout this paper, $\K$ denotes a field of characteristic zero. In addition, starting at Section \ref{sect Nash}, we also assume that $\K$ is algebraically closed.

%%%%%%%%%%%%%%%%%%%%%%%%%%%%%%%%%%%%
%		A higher-order Jacobian matrix for parametrized varieties
%%%%%%%%%%%%%%%%%%%%%%%%%%%%%%%%%%%%

\section{A higher-order Jacobian matrix of morphisms} 

%%%%%%%	    A higher-order Jacobian matrix of a morphism	

\subsection{A higher-order Jacobian matrix of a morphism of affine spaces}

In this section we study a higher-order derivative of a morphism between affine varieties and find a matrix representation of 
this linear map. 

\begin{nota}\label{main not}
The following notation will be constantly used in this paper.
\begin{itemize}
\item The entries of vectors $\alpha\in\N^t$ are denoted as $\alpha=(\alpha(1),\ldots,\alpha(t))$.
\item $\alpha\leq\beta\Leftrightarrow\alpha(i)\leq\beta(i)\mbox{ }\forall i\in\{1,\ldots,t\}$. In particular, $\alpha<\beta$ if and only if $\alpha(i)\leq\beta(i)\mbox{ }\forall i\in\{1,\ldots,t\}$ and $\alpha(i)<\beta(i)$ for some $i\in\{1,\ldots,t\}$.
\item $|\alpha|=\alpha(1)+\cdots+\alpha(t)$.
\item $\alpha!=\alpha(1)!\alpha(2)!\cdots\alpha(t)!$
\item $\partial^{\alpha}=\partial^{\alpha(1)}\partial^{\alpha(2)}\cdots\partial^{\alpha(t)}$.
\item For $t,n\in\N$, $\Ltn:=\{\gamma\in\N^t|1\leq|\gamma|\leq n\}$. In addition, we denote $\lambda_{t,n}:=|\Ltn|=\binom{n+t}{t}-1$.
\end{itemize}
\end{nota}

Consider a morphism 
\begin{align}
\vp:\K^d&\rightarrow\K^s,\notag\\
x=(x_1,\ldots,x_d)&\mapsto(g_1(x),\ldots,g_s(x)).\notag 
\end{align}
Assume that $\vp$ is regular at some $x\in\K^d$ and let $y=\vp(x)\in\K^s$. Let $\m\subset\Kx$ and $\n\subset\Ky$ be the 
maximal ideals corresponding to $x$ and $y$, and $\m_x$, $\n_y$ the maximal ideals in $(\Kx)_{\m}$ and $(\Ky)_{\n}$, 
respectively.

Let $\vp^*:(\Ky)_{\n}\rightarrow(\Kx)_{\m}$ be the induced homomorphism on local rings, where $\vps(\n_y)\subset\m_x$. 
In particular, there is a homomorphism of $\K-$vector spaces for each $n\in\N$:
\begin{align}
(\vpb)_n:\n_y/\n_y^{n+1}&\rightarrow\m_x/\m_x^{n+1}.\notag
\end{align}
Let $A_x=\{(X-x)^{\alpha}:=(X_1-x_1)^{\alpha(1)}\cdots(X_d-x_d)^{\alpha(d)}|\alpha\in\Ld\}$ be a basis 
of $\m_x/\m_x^{n+1}$ as $\K-$vector space. Similarly, let $B_y=\{(Y-y)^{\beta}|\beta\in\Ls\}$ be a basis 
of $\n_y/\n_y^{n+1}$. The dual bases of $A_x$ and $B_y$ are, respectively,
\begin{align}
A_x^{\vee}&=\Big\{\frac{1}{\alpha!}\frac{\partial^{\alpha}}{\partial X^{\alpha}}\Big |_{x}|\alpha\in\Ld\Big\},
\notag\\
B_y^{\vee}&=\Big\{\frac{1}{\beta!}\frac{\partial^{\beta}}{\partial Y^{\beta}}\Big |_{y}|\beta\in\Ls\Big\}.\notag
\end{align}
Since $(\vpb)_n((Y-y)^{\beta})=(g_1-g_1(x))^{\beta(1)}\cdots(g_s-g_s(x))^{\beta(s)}=:(\vp-\vp(x))^{\beta}$, it follows that 
the dual morphism $\vpd_n:(\m_x/\m_x^{n+1})^{\vee}\rightarrow(\n_y/\n_y^{n+1})^{\vee}$ satisfies 
\begin{align}\label{e. coeff}
\vpd_n\Big(\frac{1}{\alpha!}\frac{\partial^{\alpha}}{\partial X^{\alpha}}\Big |_{x}\Big)=
\frac{1}{\alpha!}\frac{\partial^{\alpha}}{\partial X^{\alpha}}\Big |_{x}\circ(\vpb)_n:\n_y/\n_y^{n+1}&\rightarrow\K,\\
(Y-y)^{\beta}&\mapsto \frac{1}{\alpha!}\frac{\partial^{\alpha}(\vp-\vp(x))^{\beta}}{\partial X^{\alpha}}\Big |_{x}.\notag
\end{align}
It follows that the matrix representation of $\vpd_n$ in these bases is:
\begin{equation}\label{e. jac}
\big[\vpd_n\big]_{A_x^{\vee}}^{B_y^{\vee}}=\Big(\frac{1}{\alpha!}\frac{\partial^{\alpha}(\vp-\vp(x))^{\beta}}{\partial X^{\alpha}}\Big 
|_{x}\Big)_{\beta\in\Ls,\alpha\in\Ld}.
\end{equation}

\begin{defi}\label{jac matrix}
Let $\vp:\K^d\rightarrow\K^s$ be as before, where $\vp(x)=y$. We call the linear map $\vpd_n$ the \textit{derivative of order}
$n$ \textit{of} $\vp$ \textit{at} $x$. In addition, let
$$\Dn:=\Big(\frac{1}{\alpha!}\frac{\partial^{\alpha}(\vp-\vp(x))^{\beta}}{\partial X^{\alpha}}|_x\Big)
_{\beta\in\Ls, \alpha\in\Ld}.$$
We call $\Dn$ the \textit{Jacobian matrix of order} $n$ \textit{of} $\vp$ \textit{at} $x$ or the \textit{higher-order Jacobian matrix of} 
$\vp$ \textit{at} $x$. Notice that $\Dn$ is a $(\ls\times \ld)$-matrix. We order the rows and columns of this matrix increasingly using graded lexicographical order on $\Ls$ and $\Ld$. This order is denoted $\preceq$.
%, and assuming $X_1<X_2<\cdots<X_n$ and similarly for the variables $Y_j's$. 
\end{defi}

\begin{rem}\label{r. Taylor}
Notice that, for each $\beta\in\Ls$, the $\beta$ row of $\Dn$ corresponds precisely to the coefficients of the truncated Taylor 
expansion of order $n$ of $(\vp-\vp(x))^{\beta}$ centered at $x$.
\end{rem}

\begin{rem}
A similar higher-order Jacobian matrix of a single polynomial $F$ was defined in \cite{D3} and is denoted $\Jac$. See also 
\cite{BD,BJNB} for a further development of this matrix.
\end{rem}

\begin{exam}
Let $\vp:\K\rightarrow\K^2$, $t\mapsto(t,t^2)$. The usual matrix representation of the derivative of $\vp$
at $0\in\K$ is given by the Jacobian matrix:
\[D_0(\vp)= \left( 
\begin{array}{cc}
\frac{d t}{dt}|_0\\
\frac{d t^2}{dt}|_0
\end{array} \right)=
\left( \begin{array}{ccccc}
1 \\
2t
\end{array} \right){\big|_0}.
\]
Following the construction of the higher-order Jacobian matrix given previously, in the case $n=2$, we obtain:
\[
D^2_0(\varphi)= \left( \begin{array}{ccccc}
\frac{d t}{dt}|_0&  	       		&\frac{1}{2!}\frac{d^2 t}{dt^2}|_0\\
\frac{d t^2}{dt}|_0&  	      			&\frac{1}{2!}\frac{d^2 t^2}{dt^2}|_0\\
\frac{d (t)^2}{dt}|_0&       			&\frac{1}{2!}\frac{d^2 (t)^2}{dt^2}|_0\\
\frac{d (t\cdot t^2)}{dt}|_0&    		&\frac{1}{2!}\frac{d^2 (t\cdot t^3)}{dt^2}|_0\\
\frac{d (t^2)^2}{dt}|_0&       		&\frac{1}{2!}\frac{d^2 (t^2)^2}{dt^2}|_0      
\end{array} \right)=
\left( \begin{array}{ccccc}
1&	  	       		&0\\
2t&  	      			&1\\
0&	       			&1\\
0&	  	      			&2t\\
0&   	    			&4t^2
\end{array} \right)_{\big|_0}.
\]
\end{exam}

The higher-order Jacobian matrix satisfies the following basic properties.

\begin{lem}\label{l. basic properties}
Let $\vp:\K^d\rightarrow\K^s$ be as before.
\begin{itemize}
\item[(i)] If $\beta\in\Ls$ is such that $|\beta|=1$ then 
$\frac{\partial^{\alpha}(\vp-\vp(x))^{\beta}}{\partial X^{\alpha}}=\frac{\partial^{\alpha}\vp^{\beta}}{\partial X^{\alpha}}$
for every $\alpha\in\Ld$.
\item[(ii)] Let $\alpha\in\Ld,\beta\in\Ls$ be such that $|\alpha|<|\beta|$. Then
$\frac{\partial^{\alpha}(\vp-\vp(x))^{\beta}}{\partial X^{\alpha}}|_x=0$.
\item[(iii)] $D^1_x(\vp)$ is the usual Jacobian of $\vp$ evaluated at $x$.
\item[(iv)] If $\vp:\K^d\rightarrow\K^d$ is the identity then $\Dn$ is the identity matrix.
\item[(v)] Let $\psi:\K^s\rightarrow\K^r$ be another morphism. Then $D^n_x(\psi\circ\vp)=D^n_y(\psi)\Dn$.
\end{itemize}
\end{lem}
\begin{proof}
\begin{itemize}
\item[(i)] If $\beta\in\Ls$ is such that $|\beta|=1$ then $(\vp-\vp(x))^{\beta}=g_i-g_i(x)$ for some $i\in\{1,\ldots,s\}$.
Since $g_i(x)$ is a constant, the result follows.
\item[(ii)] The hypothesis on $\alpha$ and $\beta$ means that in $\frac{\partial^{\alpha}(\vp-\vp(x))^{\beta}}{\partial X^{\alpha}}$
the order of the derivative is less than the number of factors in $(\vp-\vp(x))^{\beta}$. This implies that in every summand of 
$\frac{\partial^{\alpha}(\vp-\vp(x))^{\beta}}{\partial X^{\alpha}}$ there is a factor $g_i-g_i(x)$. Thus
$\frac{\partial^{\alpha}(\vp-\vp(x))^{\beta}}{\partial X^{\alpha}}|_x=0$.
\item[(iii)] This follows from the definition of $\Dn$ and $(i)$.
\item[(iv)] If $\vp$ is the identity then $(\vpb)_n:\m_x/\m_x^{n+1}\rightarrow\m_x/\m_x^{n+1}$ is also the identity. 
By choosing the same basis for both vector spaces we conclude that $\Dn$ is the identity matrix.
\item[(v)] We know that $(\psi\circ\vp)^*=\vp^*\circ\psi^*$. Thus, 
$(\overline{(\psi\circ\vp)^*})_n=(\bar{\vp^*})_n\circ(\bar{\psi^*})_n$. Taking duals 
$(\overline{(\psi\circ\vp)^*})_n^{\vee}=((\bar{\vp^*})_n\circ(\bar{\psi^*})_n)^{\vee}=
(\bar{\psi^*})_n^{\vee}\circ(\bar{\vp^*})_n^{\vee}$. The result follows.
\end{itemize}
\end{proof}

Now suppose that $X\subset\K^d$ and $Y\subset\K^s$ are affine varieties and let $\vp:X\rightarrow Y$ be a morphism which
is regular at $x\in X$ and let $y=\vp(x)$. Denote by $\bm_x$ and $\bn_y$ the maximal ideals of the corresponding local rings. 
Since $\vp$ is the restriction of a morphism $\vp:\K^d\rightarrow\K^s$, the diagram
$$\xymatrix{X \ar[r]^{\vp}  \ar@{^(->}[d]^i & Y \ar@{^(->}[d]^i \\  \K^d \ar[r]^{\vp} & \K^s }$$
induces the diagram
$$\xymatrix{(\bm_x/\bm_x^{n+1})^{\vee} \ar[r] \ar@{^(->}[d] & (\bn_y/\bn_y^{n+1})^{\vee}\ar@{^(->}[d]\\ 
(\m_x/\m_x^{n+1})^{\vee}\ar[r] & (\n_y/\n_y^{n+1})^{\vee}}$$
Taking bases as before we identify $(\m_x/\m_x^{n+1})^{\vee}\cong\K^{\ld}$ and $(\n_y/\n_y^{n+1})^{\vee}\cong\K^{\ls}$. 
The commutativity of the diagram
$$\xymatrix{(\bm_x/\bm_x^{n+1})^{\vee} \ar[r] \ar@{^(->}[d] & (\bn_y/\bn_y^{n+1})^{\vee}\ar@{^(->}[d]\\ 
\K^{\ld} \ar[r]^{\Dn} & \K^{\ls}}$$
allows us to define a higher-order tangent map of $\vp:X\rightarrow Y$ at $x\in X$ as the restriction 
$$\Dn:(\bm_x/\bm_x^{n+1})^{\vee}\rightarrow(\bn_y/\bn_y^{n+1})^{\vee}.$$

%%%%%%%	   Higher-order Jacobian matrix and birational morphisms

\subsection{Higher-order Jacobian matrix and birational morphisms}

Let $Y\subset\K^s$ be an irreducible algebraic variety and $y\in Y$. In this subsection, we use the higher-order Jacobian matrix 
to explicitly compute the space $(\bn_y/\bn_y^{n+1})^{\vee}$ in some cases.
%, where $\bn_y$ is the maximal ideal of the corresponding local ring.

\begin{lem}\label{l. birational}
Let $X$ and $Y$ be irreducible varieties and let $\vp:X\dashrightarrow Y$ be a birational morphism. Let 
$U\subset X$ and $V\subset Y$ be isomorphic open subsets. Let $x\in U$ and $y=\vp(x)\in V$. Then $\vp$ induces an 
isomorphism $\bn_y/\bn_y^{n+1}\cong\bm_x/\bm_x^{n+1}$.
\end{lem}
\begin{proof}
Since $\vp|_U:U\rightarrow V$ is an isomorphism, there is an induced isomorphism on local rings $\mathcal{O}_{Y,y}\cong\Oxx$.
In particular, $\vp^*(\bn_y)=\bm_x$. The result follows.
\end{proof}

\begin{pro}\label{p. image}
Let $\vp:\K^d\dashrightarrow Y\subset\K^s$ be a birational morphism, $U\subset\K^d$ and $V\subset Y$ isomorphic open
subsets, and $y=\vp(x)$ for some $x\in U$. Then the vector space $(\bn_y/\bn_y^{n+1})^{\vee}$ is isomorphic to 
the image of the linear map defined by $\Dn$. In particular, $rank(\Dn)=\ld$.
\end{pro}
\begin{proof}
We have the following commutative diagram
$$\xymatrix{\K^d \ar@{-->}[r]^{i\circ\vp} \ar@{-->}[dr]_{\vp} & \K^s \\ & Y \ar@{^(->}[u]_i}$$
This diagram induces in turn the following commutative diagram
$$\xymatrix{(\m_x/\m_x^{n+1})^{\vee} \ar[r]^{\overline{(i\circ\vp)^*}^{\vee}} \ar[dr]_{\vpd}^{\cong} 
& (\n_y/\n_y^{n+1})^{\vee}\\ 
& (\bn_y/\bn_y^{n+1})^{\vee}, \ar@{^(->}[u]_{(\bar{i^*})^{\vee}}}$$
where the isomorphism in the diagonal arrow comes from lemma \ref{l. birational}. 
Fixing bases for $\m_x/\m_x^{n+1}$ and $\n_y/\n_y^{n+1}$ as in the previous section, we identify 
$(\m_x/\m_x^{n+1})^{\vee}\cong\K^{\ld}$ and $(\n_y/\n_y^{n+1})^{\vee}\cong\K^{\ls}$, 
In addition, from (\ref{e. jac}), it follows that $\overline{(i\circ\vp)^*}^{\vee}$ is the linear map 
defined by the matrix $D_x^n(i\circ\vp)=\Dn$. We thus obtain the diagram
$$\xymatrix{\K^{\ld}\ar[r]^{\Dn} \ar[dr] & \K^{\ls}\\ 
& (\bn_y/\bn_y^{n+1})^{\vee}. \ar@{^(->}[u]}$$
The commutativity of this diagram proves the proposition.
\end{proof}

\begin{rem}
Notice that the proofs in this section considers local rings of points of a variety. Therefore, these results are also valid 
in the analytic case. In particular, we can define a higher-order Jacobian matrix for germs of analytic maps 
$\vp:(X,x)\rightarrow(Y,y)$. 
\end{rem}

\begin{exam}
Let $\vp:\K\rightarrow C=\V(y-x^2)\subset\K^2$, $t\mapsto(t,t^2)$. We computed $D_0^2(\vp):\K^2\rightarrow\K^5$ 
in the previous section. Let $\bn_0$ be the maximal ideal of $(0,0)\in C$. Using proposition \ref{p. image} we obtain
\[(\bn_0/\bn_0^{3})^{\vee}=Im(D_0^2(\vp))=
Im\left( 
\begin{array}{ccccc}
1& 	 	       		&0\\
2t&  	      			&1\\
0&	       			&1\\
0&  		      			&2t\\
0&	       			&4t^2
\end{array} 
\right)_{\big|_0}\subset\K^5.
\]
\end{exam}

%%%%%%%%%%%%%%%%%%%%%%%%%%%%%%%%%%%%
%		Higher Jacobian monomial morphisms
%%%%%%%%%%%%%%%%%%%%%%%%%%%%%%%%%%%%

\section{Higher order Jacobian matrix for monomial morphisms}\label{s. monomial morphism}

Let $a_1,\ldots,a_s\in\Z^d$. We assume that $d\leq s$. In this section we study the higher-order Jacobian matrix of the monomial morphism 
\begin{align}\label{monomial morph}
\vp:(\K\setminus\{0\})^d&\rightarrow\K^s \\
x=(x_1,\ldots,x_d)&\mapsto(x^{a_1},\ldots,x^{a_s}),\notag
\end{align}
where $x^{a_i}:=x_1^{a_i(1)}\cdots x_d^{a_i(d)}$. 

\begin{nota}
The following notation will be also used constantly.
\begin{itemize}
\item $A$ denotes the $(d\times s)$-matrix whose columns are the vectors $a_1,\ldots,a_s$. By abuse of notation, the set 
$\{a_1,\ldots,a_s\}$ is also denoted as $A$.
\item $A_i:=(a_1(i),\ldots,a_s(i))$, $i=1,\ldots,d$, denote the rows of $A$. In parti-cular, for $\gamma\in\N^s$,
$$X^{A\gamma}=X_1^{A_1\cdot\gamma}\cdots X_d^{A_d\cdot\gamma},$$
where $A\gamma$ is a product of matrices and $A_i\cdot\gamma$ is the usual inner product in $\R^s$.
\item For $\beta\in\N^s$, we denote
$$(X^A-x^A)^{\beta}:=(X^{a_1}-x^{a_1})^{\beta(1)}\cdots(X^{a_s}-x^{a_s})^{\beta(s)}.$$
\item For $\lambda,\tau\in\N^t$, 
denote $\binom{\lambda}{\tau}:=\binom{\lambda(1)}{\tau(1)}\cdots\binom{\lambda(t)}{\tau(t)}.$
\end{itemize}
\end{nota}

With this notation, the higher-order Jacobian of $\vp$ at a point $x\in(\K\setminus\{0\})^d$ is given by:
$$\Dn=\Big(\frac{1}{\alpha!}\frac{\partial^{\alpha}(X^A-x^A)^{\beta}}{\partial X^{\alpha}}|_{x}\Big)
_{\beta\in\Ls, \alpha\in\Ld}.$$
We are interested in computing the maximal minors of this matrix. This will be done in several steps.

%For the following lemmas we need another multi-index notation concerning binomial coefficients. For $\lambda,\tau\in\N^t$, 
%denote $\binom{\lambda}{\tau}:=\binom{\lambda(1)}{\tau(1)}\cdots\binom{\lambda(t)}{\tau(t)}.$

\begin{lem}\label{l. deriv mon}
Let $\gamma\in\N^s$ and $\alpha\in\N^d$. Then 
$$\frac{1}{\alpha!}\frac{\partial^{\alpha}(X^{A\gamma})}{\partial X^{\alpha}}=\binom{A\gamma}{\alpha}X^{A\gamma-\alpha}.$$
%where $b_{\gamma,\alpha}:=\prod_{i=1}^{d}\prod_{j=0}^{\alpha(i)-1}(A_i\cdot\gamma-j)$.
\end{lem}
\begin{proof}
This is a direct computation.
\end{proof}

%%%%%%%%%%%%%%%%%%%%%%%%%%%%%nueva referencia
%%%%%%%%%%%%%%%%%%%
\begin{lem}\label{coeficientes}
Let $\beta\in\Ls$, $\alpha\in\Ld$ and $x\in(\K\setminus\{0\})^d$. Then
$$\frac{1}{\alpha!}\partial^{\alpha}(X^A-x^A)^{\beta}|_{x}=c_{\beta,\alpha}x^{A\beta-\alpha},$$
where $c_{\beta,\alpha}:=\sum_{\gamma\leq\beta,\gamma\neq0}
(-1)^{|\beta-\gamma|}\binom{\beta}{\gamma}\binom{A\gamma}{\alpha}$.
\end{lem}
\begin{proof}
From the binomial theorem we obtain, for each $i\in\{1,\ldots,s\}$:
$$(X^{a_i}-x^{a_i})^{\beta(i)}=\sum_{\gamma(i)=0}^{\beta(i)}(-1)^{\beta(i)-\gamma(i)}\binom{\beta(i)}{\gamma(i)}
(X^{a_i})^{\gamma(i)}(x^{a_i})^{\beta(i)-\gamma(i)}.$$
Thus, letting $\gamma:=(\gamma(1),\ldots,\gamma(s))$,
\begin{align}
(X^A-x^A)^{\beta}&=\sum_{\gamma(1)=0}^{\beta(1)}\cdots\sum_{\gamma(s)=0}^{\beta(s)}
(-1)^{|\beta-\gamma|}\binom{\beta}{\gamma}\prod_{i=1}^{s}(X^{a_i})^{\gamma(i)}(x^{a_i})^{\beta(i)-\gamma(i)}\notag\\
&=\sum_{\gamma\leq\beta}(-1)^{|\beta-\gamma|}\binom{\beta}{\gamma}
(X^{\sum\gamma(i)a_i})(x^{\sum\beta(i)a_i-\sum\gamma(i)a_i})\notag\\
&=\sum_{\gamma\leq\beta}(-1)^{|\beta-\gamma|}\binom{\beta}{\gamma}
(X^{A\gamma})(x^{A\beta-A\gamma}).\notag
\end{align}
With this formula and the previous lemma now it is easy to compute the derivative evaluated at $x$:
\begin{align}
\frac{1}{\alpha!}\partial^{\alpha}(X^A-x^A)^{\beta}|_{x}&=\sum_{\gamma\leq\beta,\gamma\neq0}(-1)^{|\beta-\gamma|}
\binom{\beta}{\gamma}(x^{A\beta-A\gamma})\frac{1}{\alpha!}\partial^{\alpha}(X^{A\gamma})|_x\notag\\
&=\sum_{\gamma\leq\beta,\gamma\neq0}(-1)^{|\beta-\gamma|}\binom{\beta}{\gamma}(x^{A\beta-A\gamma})
\binom{A\gamma}{\alpha}X^{A\gamma-\alpha}|_x\notag\\
&=\sum_{\gamma\leq\beta,\gamma\neq0}(-1)^{|\beta-\gamma|}\binom{\beta}{\gamma}
\binom{A\gamma}{\alpha}x^{A\beta-\alpha}\notag\\
&=\Big[\sum_{\gamma\leq\beta,\gamma\neq0}(-1)^{|\beta-\gamma|}\binom{\beta}{\gamma}
\binom{A\gamma}{\alpha}\Big]x^{A\beta-\alpha}.\notag
\end{align}
\end{proof}

Using this lemma it follows that the higher-order Jacobian of $\vp$ at each $x\in(\K\setminus\{0\})^d$ has the following shape:
\begin{equation}\label{e. Jnp monomial phi}
\Dn=\Big(c_{\beta,\alpha}x^{A\beta-\alpha}\Big)_{\beta\in\Ls, \alpha\in\Ld}.
\end{equation}

\begin{pro}\label{p. minors}
Let $J=\{\beta_1,\ldots,\beta_{\ld}\}\subset\Ls$, where $\beta_1\prec\ldots\prec\beta_{\ld}$ (see definition \ref{jac matrix} for the 
notation $\prec$). Let $L_J$ denote the submatrix of $\Dn$ formed by the rows $\beta_1,\ldots,\beta_{\ld}$ and all of its columns 
$\alpha_1,\ldots,\alpha_{\ld}$. Then, if $x\in(\K\setminus\{0\})^d$, 
$$\det(L_J)=\frac{x^{A\beta_1+\cdots+A\beta_{\ld}}}{x^{\alpha_1+\cdots+\alpha_{\ld}}}\det(L_J^c),$$
where $L_J^c:=(c_{\beta_i,\alpha_j})_{i,j}$.
\end{pro}
\begin{proof}
The matrix whose determinant we want to compute is the following:
\[L_J=
\left( \begin{array}{cccc}
&c_{\beta_1,\alpha_1}x^{A\beta_1-\alpha_1}			&\cdots		&c_{\beta_1,\alpha_{\ld}}x^{A\beta_1-\alpha_{\ld}}\\
&c_{\beta_2,\alpha_1}x^{A\beta_2-\alpha_1}			&\cdots		&c_{\beta_2,\alpha_{\ld}}x^{A\beta_2-\alpha_{\ld}}\\
&\vdots  											&\cdots	 	&\vdots\\
&c_{\beta_{\ld},\alpha_1}x^{A\beta_{\ld}-\alpha_1}		&\cdots		&c_{\beta_{\ld},\alpha_{\ld}}x^{A\beta_{\ld}-\alpha_{\ld}}
\end{array} \right).
\]
Multiply the $\alpha_j th$ column by $x^{\alpha_j}$. Then multiply the $\beta_i th$ row by $x^{-A\beta_i}$. 
Let $L_J^c=(c_{\beta_i,\alpha_j})_{i,j}$. Then
\begin{equation}\label{e. minors}
\det(L_J)=\frac{x^{A\beta_1+\cdots+A\beta_{\ld}}}{x^{\alpha_1+\cdots+\alpha_{\ld}}}\det(L_J^c).
\end{equation}
\end{proof}

\begin{rem}\label{r. Taylor coef}
If $n=1$ then $\ld=d$. Letting $\beta_{i_k}=e_{i_k}\in\N^s$ for $k=1,\ldots,d$ and 
$J=\{\beta_{i_1},\ldots,\beta_{i_d}\}\subset\Lambda_{s,1}$, it follows that $L_J^c$ is the $(d\times d)$-matrix whose rows are 
$a_{i_1},\ldots,a_{i_d}$. In particular, in view of (\ref{e. minors}), $\det(L_J)\neq0$ if and only if $a_{i_1},\ldots,a_{i_d}$ are 
linearly independent. This remark allows a comparison between the so-called logarithmic Jacobian ideal of a toric variety and an ideal whose blowup defines the Nash blowup of the variety \cite{GS-1,No}. This, in turn, gives place to the fact that the Nash blowup of a toric variety can be obtained as the blowup of its logarithmic Jacobian ideal (see \cite{GS-1,LJ-R,GT}). As a result, there is an explicit combinatorial description of the Nash blowup in this context \cite{GS-1,GT,GM}.
\end{rem}

%%%%%%%%%%%%%%%%%%%%%%%%%%%%%%%%%%%%
%		Higher Nash blowup toric varieties
%%%%%%%%%%%%%%%%%%%%%%%%%%%%%%%%%%%%

\section{Higher Nash blowup of toric varieties}\label{sect Nash}

In this section we exhibit an open cover for the higher Nash blowup of a toric variety. The main result is the first step 
toward our study of a conjecture proposed by T. Yasuda regarding the higher Nash blowup of formal curves.

We start by recalling the definition of the higher Nash blowup of an algebraic variety. Subsection \ref{s. cover} is based on the general theory of (not necessarily normal) toric varieties developed in \cite{GT} and also uses some ideas appearing in \cite{GM}.

\subsection{Higher Nash blowup}

\begin{nota}
Given an irreducible algebraic variety $X\subset\K^s$ of dimension $d$ and a point $x\in X$, we denote 
$\Tx:=(\bm_x/\bm_x^{n+1})^{\vee}$. This is a vector space of dimension $\ld$, whenever $x$ is a non-singular point.
\end{nota}

Notice that $X\subset\K^s$ implies $\Tx\subset T^n_x\K^s\cong\K^{\ls}$. Thus, if $x$ is a non-singular point, we can see $\Tx$ 
as an element of the Grassmanian $Gr(\ld,\K^{\ls})$.

\begin{defi}\cite{No, OZ, Y1}\label{d. Nash}
Let $X\subset\K^s$ be an irreducible algebraic variety of dimension $d$. Consider the \textit{Gauss map of order} $n$:
\begin{align}
G_n:X\setminus Sing(X&)\rightarrow Gr(\ld,\K^{\ls})\notag\\
x&\mapsto\Tx,\notag
\end{align}
where $Sing(X)$ denotes the set of singular points of $X$. Denote by $\Na$ the Zariski closure of the graph of $G_n$. 
Call $\pi_n$ the restriction to $\Na$ of the projection of $X\times Gr(\ld,\K^{\ls})$ to $X$. The pair $(\Na,\pi_n)$ is called 
the \textit{higher Nash blowup} of $X$ or the \textit{Nash blowup of X of order n}.
\end{defi}

It was proposed by T. Yasuda (\cite{Y1}) to solve the singularities of $X$ by applying once the higher Nash blowup for $n$ 
sufficiently large. Yasuda himself proved that his method works for curves (\cite[Corollary 3.7]{Y1}). Moreover, Yasuda suggested
in \cite[Remark 1.5]{Y3} that the $A_3$-singularity might be a counterexample to his conjecture on the one-step resolution. 
R. Toh-Yama recently proved in \cite{T} that $Nash_n(A_3)$ is singular for every $n\geq1$.

In addition to the previous conjecture, Yasuda also proposed another one concerning the numerical semigroup 
associated to the higher Nash blowup of formal curves (see conjecture \ref{conj} below). The results 
we obtain in this section will be used to study this conjecture in the case of toric curves.

%%%%%	  Open cover higher Nash

\subsection{An explicit open cover of the higher Nash blowup of toric varieties by affine toric varieties}\label{s. cover}

Let us recall the definition of an affine toric variety (see, for instance, \cite[Section 1.1]{CLS} or \cite[Chapter 4]{St}). 

\begin{defi}\label{d. toric}
Let $A=\{a_1,\ldots,a_s\}\subset\Z^d$. Let $\Ga:=\N A$ denote the semigroup generated by $A$, i.e., $\Ga=\{\sum_i\lambda_ia_i|\lambda_i\in\N\}$. In addition, assume that $\Z A=\{\sum_i\lambda_i a_i|\lambda_i\in\Z\}=\Z^d$. Consider the following monomial morphism:
\begin{align}\label{e. mon map}
\vpg:(\K^*)^d&\rightarrow \K^s\\
x=(x_1,\ldots,x_d)&\mapsto (x^{a_1},\ldots,x^{a_s}),\notag
\end{align}
where $\K^*=\K\setminus\{0\}$. Let $X_{\Gamma}$ denote the Zariski closure of the image of $\vpg$. We call $X_{\Gamma}$ the \textit{affine toric variety} defined by $\Gamma$.
\end{defi}

It is well known that $X_{\Gamma}$ is an irreducible variety of dimension $d$, contains a dense open set isomorphic to $(\K^*)^d$ 
and such that the natural action of $(\K^*)^d$ on itself extends to an action on the variety. In addition, $\Xg$ does not depend on the generating set $A$ (see \cite[Theorem 1.1.17]{CLS} for various equivalent characterizations of affine toric varieties).

\begin{pro}\cite[Prop. 1.2.12]{CLS},\cite[Prop. 15]{GT}\label{0 dim orbit}
Let $X_{\Ga}\subset\K^s$ be an affine toric variety, $\cond:=\R_{\geq0}\Gamma\subset\R^d$ the cone generated by $\Ga$, and $\sigma$ its dual cone. The following statements are equivalent:
\begin{itemize}
\item[(a)] $0\in\Xg$.
\item[(b)] $\Xg$ has a 0-dimensional orbit.
\item[(c)] The cone $\con$ is of dimension $d$.
\item[(d)] The cone $\cond$ is strongly convex.
\end{itemize}
\end{pro}

We want to show that the higher Nash blowup of a toric variety having a 0-dimensional orbit, has a finite open cover given by affine toric varieties with the same property. The proof of this fact is based on the following combinatorial construction of blowing ups of monomials ideals in toric varieties (see \cite[Section 2.6]{GT}).
\\
\\
\noindent\textbf{Combinatorial description of the blowup of a monomial ideal.} Let $\Xg\subset\K^s$ be an affine toric variety having a 0-dimensional orbit and $\cond=\R_{\geq0}\Ga\subset\R^d$ (which is strongly convex, by the previous proposition).
\begin{itemize}
\item[(i)] Let $I=\langle x^m |m\in B\rangle\subset\K[\Xg]$ be a monomial ideal.
\item[(ii)] Let $\NI$ be the Newton polyhedron of $I$, i.e., the convex hull in $\R^d$ of the set $\{m+\cond|m\in B\}$.
\item[(iii)] Let $m'\in B$. Denote $\Ga_{m'}:=\Ga+\N(\{m-m'|m\in B\}).$
\item[(iv)] Given $m',m''\in B$, the affine toric varieties $X_{\Gamma_{m'}}$ and $X_{\Gamma_{m''}}$ can be glued together along the principal open subsets $X_{\Gamma_{m'}}\setminus\V(x^{m''-m'})$ and $X_{\Gamma_{m''}}\setminus\V(x^{m'-m''})$. There is an  isomorphism between these open subsets which is induced by localizations of coordinate rings:
$$\K[X_{\Gamma_{m'}}]_{\frac{x^{m''}}{x^{m'}}}
\cong \K[X_{\Gamma_{m''}}]_{\frac{x^{m'}}{x^{m''}}}.$$
\item[(v)] The variety resulting from the previous glueing is the blowup of $\Xg$ along $I$ (see \cite[Proposition 32]{GT}). We denote it as $Bl_I\Xg$.
\item[(vi)] Finally, let $B'=\{m'\in B|m'\mbox{ is a vertex of }\NI\}$. Then 
$$Bl_I\Xg=\quad \quot{\bigsqcup_{m' \in B'}X_{\Gamma_{m'}}}{\sim}$$
(see the proof of Proposition 32, \cite{GT}). By proposition \ref{0 dim orbit}, for $m'\in B'$, $X_{\Ga_{m'}}$ has a 0-dimensional orbit. In particular, $Bl_I\Xg$ has an open cover by affine toric varieties having a 0-dimensional orbit.
\end{itemize}

\begin{rem}
The variety resulting from the previous construction is an example of an abstract toric variety having a good action (see \cite[Section 2.8]{GT}). These varieties are characterized by the fact that they can be described in combinatorial terms by families of semigroups labeled by fans (see \cite[Theorem 44]{GT}).
\end{rem}

In order to use the previous construction and compare it to the higher Nash blowup of a toric variety, we need to introduce some monomial ideal. In addition, we use the Pl\"ucker embedding of $Gr(\ld,\K^{\ls})$ into the projective space $\Pro$.  First, some notation.

\begin{nota}\label{e. S}
Let $A=\{a_1,\ldots,a_s\}\subset\Z^d$ and $\Gamma=\N A$ a semigroup defining a toric variety $\Xg\subset\K^s$. 
\begin{itemize}
\item Given $J=\{\beta_1,\ldots,\beta_{\ld}\}\subset\Ls$ such that $\beta_1\prec\cdots\prec\beta_{\ld}$, we denote as $U_J$ the affine
chart of $\Pro$ where the $J$-coordinate is non-zero (see definition \ref{jac matrix} for the notation $\prec$).
\item Let $S_A:=\{J=\{\beta_1,\ldots,\beta_{\ld}\}\subset\Ls|\beta_1\prec\cdots\prec\beta_{\ld},\det(L_J^c)\neq0\}.$ 
Notice that $S_A\neq\emptyset$ by propositions \ref{p. image} and \ref{p. minors}.
\item For each $J=\{\beta_1,\ldots,\beta_{\ld}\}\subset \Lambda_{s,n}$, denote $m_J:=A\beta_1+\cdots+A\beta_{\ld}.$ 
\end{itemize}
\end{nota}

\begin{defi}\label{log jac n}
Let $\I_n:=\langle X^{m_J}|J\in S_A\rangle\subset\K[\Xg]$. Following the usual terminology, we call $\I_n$ the \textit{logarithmic Jacobian ideal of order n} of $\Xg$.
\end{defi}

\begin{rem}
In the following subsection we show that $\I_n$ does not depend on the set of generators of $\Gamma$.
\end{rem}

We want to apply the combinatorial description of the blowup of a monomial ideal to $\I_n$. To that end, we simplify a little the notation coming from that description. For $X^{m_{J}}\in\I_n$, instead of using $\Gamma_{m_{J}}$ as in (iii), we simply write $\Ga_{J}$.

Now we are ready to prove the main theorem of this section.

\begin{teo}\label{main s3}
Let $\Xg\subset\K^s$ be an affine toric variety having a 0-dimensional orbit. Then $\Nag$ is isomorphic to the blowup of the logarithmic Jacobian ideal of order $n$ of $\Xg$. In particular, $\Nag$ has a finite open covering given by affine toric varieties having a 0-dimensional orbit.
\end{teo}
\begin{proof}
We divide this proof into two steps: the first one describes locally $\Nag$ and the second one is a glueing argument.
\\

\noindent\textit{Step I: } According to proposition \ref{p. image} and (\ref{e. Jnp monomial phi}), for a point 
$p:=\vpg(x)\in\Xg$, for some $x\in(\K^*)^d$, we have
$$\Tpg=Im(\Dng)=Im\Big(c_{\beta,\alpha}x^{A\beta-\alpha}\Big)_{\beta\in\Ls, \alpha\in\Ld}.$$
Thus, the Pl\"ucker coordinates of $\Tpg\in Gr(\ld,\K^{\ls})\hookrightarrow\Pro$ are given by the maximal minors of 
$\big(c_{\beta,\alpha}x^{A\beta-\alpha}\big)_{\beta,\alpha}$. According to (\ref{e. minors}), for a choice 
$J=\{\beta_1,\ldots,\beta_{\ld}\}\subset\Ls$, where $\beta_1\prec\ldots\prec\beta_{\ld}$, the corresponding minor is:
$$\det(L_J^c)\frac{x^{A\beta_1+\cdots+A\beta_{\ld}}}{x^{\alpha_1+\cdots+\alpha_{\ld}}}.$$
Fix $J_0\in S_A$. It follows that: 
\begin{enumerate}
\item If $J\in S_A$ we can make a change of coordinates in $U_{J_0}\cong\K^{s+\binom{\ls}{\ld}-1}$ to turn the non-zero 
constant $\frac{\det(L_J^c)}{\det(L_{J_0}^c)}$ into 1. Thus, we can assume that the $J$-coordinate of $\Nag\cap U_{J_0}$ 
comes with the constant 1, for every $J\in S_A$.
\item If $J\notin S_A$ the $J$-coordinate of $\Nag$ is zero. This implies that we can embed $\Nag\cap U_{J_0}$ 
in $\K^{s+|S_A|-1}$.
\end{enumerate}
These two remarks imply that
\begin{align}\label{e. af chart}
\Nag\cap U_{J_0}&\cong\overline{\{\Big(\vpg(x),\frac{x^{\sum_{\beta_i\in J}A\beta_i}}
{x^{\sum_{\beta^0_i\in J_0}A\beta^0_i}}\Big)|J\in S_A\setminus\{J_0\},x\in(\K^*)^d\}}\notag\\
&=\overline{\{(\vpg(x),x^{m_J-m_{J_0}})|J\in S_A\setminus\{J_0\},
x\in(\K^*)^d\}}\\
&=\overline{Im(\varphi_{\Ga_{J_0}})}\subset\K^{s+|S_A|-1}.\notag
\end{align}
In particular, this affine chart of $\Nag$ is an affine toric variety.
\\

\noindent\textit{Step II: } By Step I, for each $J\in S_A$, $X_{\Gamma_J}\cong\Nag\cap U_{J}$. 
Since both $Bl_{\I_n}\Xg$ and $\Nag$ are obtained by glueing $X_{\Gamma_J}$ and $\Nag\cap U_{J}$, respectively, we only need to 
check that the glueing is the same. The glueing in $\Nag\subset\Xg\times\Pro$ is given by the usual glueing in $\Pro$, i.e., 
the one induced by the following isomorphisms of localizations of coordinate rings for each couple $J_1,J_2\in S_A$: 
\begin{align}
\K[x^{a_1},\ldots,x^{a_s},x^{m_J-m_{J_1}}&|J\in S_A\notin\{J_1\}]_{\frac{x^{m_{J_2}}}{x^{m_{J_1}}}}\notag\\
&\cong \K[x^{a_1},\ldots,x^{a_s},x^{m_J-m_{J_2}}|J\in S_A\notin\{J_2\}]_{\frac{x^{m_{J_1}}}{x^{m_{J_2}}}}.\notag
\end{align}
This is exactly the glueing described in the combinatorial description of the blowup of a monomial ideal.
\end{proof}

\begin{rem}
For $n=1$, the previous theorem was proved in \cite{GS-1,LJ-R,GT}.
\end{rem}

\begin{rem}
The previous theorem and its proof show that $\Nag$ can be covered by open affine varieties which are invariant under the action of a torus. This statement could be obtained directly using results of \cite{GT, Y1}. Indeed, by \cite[Section 2.2]{Y1}, the higher Nash blowup of 
a toric variety is an equivariant morphism; in particular, it is the blowup of some monomial ideal. Then \cite[Corollary 34]{GT} implies the statement. We want to emphasize that the contribution of this section is that one can take the logarithmic Jacobian ideal of order $n$ as such monomial ideal. In addition, we describe an explicit method to construct this ideal. 
\end{rem}

%%%%% Logarithmic Jacobian ideal

\subsection{The logarithmic Jacobian ideal of order $n$ is independent of the generators of $\Gamma$}

In this subsection we show that the ideal $\I_n$ does not depend on the set of generators $A$ of $\Gamma$. To that end, we need 
to modify temporarily the notation $\I_n$. We denote as $\I_C$ the logarithmic Jacobian ideal of order $n$, where $C$ is an arbitrary 
set of generators of $\Gamma$.

\begin{teo}\label{log jac}
Let $A=\{a_1,\ldots,a_s\}\subset\Z^d$ and $B=\{b_1,\ldots,b_t\}\subset\Z^d$ be such that $\Gamma=\N A=\N B$. Then
$\I_A=\I_{A\cup B}=\I_B.$
In particular, the logarithmic Jacobian ideal of order $n$ of $\Xg$ does not depend on the generators of $\Ga$.
\end{teo}
\begin{proof}
It is enough to show $\I_A=\I_{A\cup B}$. Lemma \ref{inc 1} states that $\I_A\subset\I_{A\cup B}$. Applying repeatedly lemma \ref{inc 2} we obtain the other inclusion.
\end{proof}

\begin{lem}\label{inc 1}
With the notation of theorem \ref{log jac}, $\I_A\subset\I_{A\cup B}$.
\end{lem}
\begin{proof}
For $J\in S_A$, define $\bar{J}:=\{(\beta,0,\ldots,0)\in\N^{s+t}|\beta\in J\}.$ The submatrix of $D_x^n(\vp_{A\cup B})$ defined 
by $\bar{J}$ is the same as the submatrix of $D_x^n(\vpg)$ defined by $J$. Therefore $\bar{J}\in S_{A\cup B}$. Thus,
$X^{m_{J}}=X^{m_{\bar{J}}}\in\I_{A\cup B}.$
\end{proof}

\begin{lem}\label{inc 2}
Let $A$ be as in theorem \ref{log jac} and $b\in\N A$. Let $A'=A\cup\{b\}$. Then $\I_{A'}\subset\I_A$.
\end{lem}
\begin{proof}
Consider the following partition of $S_{A'}$:
\begin{align}
S_1&:=\{\bar{J}\in S_{A'}|\beta(s+1)=0 \mbox{ for all } \beta\in\bar{J}\},\notag\\
S_2&:=\{\bar{J}\in S_{A'}|\beta(s+1)>0 \mbox{ for some } \beta\in\bar{J}\}.\notag
\end{align}
By definition, $\I_{A'}=\langle\{X^{m_{\bar{J}}}|\bar{J}\in S_1\}\cup \{X^{m_{\bar{J}}}|\bar{J}\in S_2\}\rangle$. 
As in the proof of lemma \ref{inc 1}, $\{X^{m_{\bar{J}}}|\bar{J}\in S_1\}\subset\I_A$. We claim that 
$\{X^{m_{\bar{J}}}|\bar{J}\in S_2\}\subset \langle \{X^{m_{\bar{J}}}|\bar{J}\in S_1\}\rangle$, implying the lemma. Now, to prove the claim we show that for $\bar{J}\in S_2$ there exists $J\in S_1$ and $\bar{\gamma}\in\Ga$ such that $m_{\bar{J}}=m_J+\bar{\gamma}$. First, we need some notation.
\begin{itemize}
\item For $\gamma\leq \beta_i$, let $\epsilon_{\gamma}:=(-1)^{|\beta_i-\gamma|}\binom{\beta_i}{\gamma}$. Then, by definition,  
$c_{\beta_i,\alpha_j}=\sum_{\gamma\leq\beta_i,\gamma\neq0}\epsilon_{\gamma}\binom{A'\gamma}{\alpha_j}$ (see lemma \ref{coeficientes}).
\item $c_{\beta_i}:=\Big(\sum_{\gamma\leq\beta_i,\gamma\neq0}\epsilon_{\gamma}\binom{A'\gamma}{\alpha_j}\Big)_{1\leq j \leq \ld} (c_{\beta_i}\mbox{ is the }\beta_i\mbox{-th row of } L^c_{\bar{J}}).$
\item $v_{\gamma}:=\Big(\binom{A'\gamma}{\alpha_j}\Big)_{1\leq j \leq \ld}$. Notice that by remark \ref{delta}, $c_{\beta_i}=\sum_{\gamma\leq\beta_i,\gamma\neq0}\epsilon_{\gamma}v_{\gamma}$ (that remark is stated for toric curves but it also holds for toric varieties of any dimension).
\end{itemize}

Let $\bar{J}=\{\beta_1,\ldots,\beta_{\ld}\}\in S_2$. Then $\det(L_{\bar{J}}^c)\neq0$ and we can assume that $\beta_1(s+1)>0$. Then the following holds:

\begin{enumerate}  
\item There exists $\gamma'\leq\beta_1$ such that the matrix obtained by replacing the $\beta_1$th row of $L_{\bar{J}}^c$ by $v_{\gamma'}$ has non-zero determinant.
\item There exists $\delta_0\in\N^{s+1}$ such that $\delta_0(s+1)=0$ and $A'\delta_0=A'\gamma'$.
\item There exists $\delta\in\N^{s+1}$ such that $\delta\leq\delta_0$, $\delta(s+1)=0$, and the matrix having as rows $c_{\delta}, c_{\beta_2},\ldots,c_{\beta_{\ld}}$ has non-zero determinant.
\item Let $J_1:=\bar{J}\setminus\{\beta_1\}\cup\{\delta\}$. Then $J_1\in S_{A'}$ and $m_{\bar{J}}$ equals $m_{J_1}$ plus some element in $\Gamma$.
\end{enumerate}

Notice that by applying 1 - 4 to any element of $\bar{J}$ whose $(s+1)$-entry is greater than zero, we obtain $J\in S_1$ and $\bar{\gamma}\in\Gamma$ with the desired properties. Now we prove the previous statements.

\begin{enumerate}
\item It follows immediately from:
\[
0\neq\det(L_{\bar{J}}^c)=
\det\left( 
\begin{array}{cccc}
c_{\beta_1} \\
c_{\beta_2} \\
\vdots\\
c_{\beta_{\ld}}
\end{array} 
\right)=
\det\left( 
\begin{array}{cccc}
\sum_{\gamma\leq\beta_1,\gamma\neq0}\epsilon_{\gamma}v_{\gamma}      		\\
c_{\beta_2} 																\\
\vdots 																	\\
c_{\beta_{\ld}}	  	      		
\end{array} 
\right)
\]
\[
=\sum_{\gamma\leq\beta_1,\gamma\neq0}\epsilon_{\gamma}\det\left( 
\begin{array}{cccc}
v_{\gamma}      		\\
c_{\beta_2} 																\\
\vdots 																	\\
c_{\beta_{\ld}}	  	      		
\end{array} 
\right).
\]

\item If $\gamma'(s+1)=0$, let $\delta_0:=\gamma'$. Now suppose that $\gamma'(s+1)=k>0$. Since $b\in\N A$, $b=\sum_{l=1}^s\lambda_la_l$. Let $\delta_0(l):=\gamma'(l)+k\lambda_l$ for $l<s+1$ and $\delta_0(s+1)=0$. Then
$$A'\delta_0=\sum_{l=1}^s\delta_0(l)a_l=\sum_{l=1}^s(\gamma'(l)+k\lambda_l)a_l=\sum_{l=1}^s\gamma'(l)a_l+kb=A'\gamma'.$$
\item Let $M$ denote the matrix whose rows are $c_{\delta_0}, c_{\beta_2},\ldots,c_{\beta_{\ld}}$, in this order.
If $\det(M)\neq0$ let $\delta:=\delta_0$. Suppose that $\det(M)=0$. Then
\[
0=\det(M)=
\sum_{\gamma<\delta_0,\gamma\neq0}\epsilon_{\gamma}\det\left( 
\begin{array}{cccc}
v_{\gamma}      		\\
c_{\beta_2} 																\\
\vdots 																	\\
c_{\beta_{\ld}}	  	      		
\end{array} 
\right)+
\det\left( 
\begin{array}{cccc}
v_{\delta_0}      		\\
c_{\beta_2} 																\\
\vdots 																	\\
c_{\beta_{\ld}}	  	      		
\end{array} 
\right).
\]
On the other hand, $A'\delta_0=A'\gamma'$ implies $v_{\gamma'}=v_{\delta_0}$ and so
\[
0\neq
\det\left( 
\begin{array}{cccc}
v_{\gamma'}      		\\
c_{\beta_2} 																\\
\vdots 																	\\
c_{\beta_{\ld}}	  	      		
\end{array} 
\right)=
\det\left( 
\begin{array}{cccc}
v_{\delta_0}      		\\
c_{\beta_2} 																\\
\vdots 																	\\
c_{\beta_{\ld}}	  	      		
\end{array} 
\right).
\] 
Therefore
\[
0\neq
\sum_{\gamma<\delta_0,\gamma\neq0}\epsilon_{\gamma}\det\left( 
\begin{array}{cccc}
v_{\gamma}      		\\
c_{\beta_2} 																\\
\vdots 																	\\
c_{\beta_{\ld}}	  	      		
\end{array} 
\right).
\]
Thus there exists $\delta_1<\delta_0$ such that $\det(v_{\delta_1}\mbox{ }c_{\beta_2}\mbox{ }\cdots\mbox{ }c_{\beta_{\ld}})\neq0$. 
If $\det(c_{\delta_1}\mbox{ }c_{\beta_2}\mbox{ }\cdots\mbox{ }c_{\beta_{\ld}})\neq0$, let $\delta:=\delta_1$. Otherwise repeat the previous process. This leads to a sequence $\delta_0>\delta_1>\cdots$. Since this
sequence cannot decrease infinitely many times, we conclude that there exists $k\geq0$ such that $\delta_0>\delta_1>\cdots>\delta_k=:\delta$ and $\det(c_{\delta}\mbox{ }c_{\beta_2}\mbox{ }\cdots\mbox{ }c_{\beta_{\ld}})\neq0$. In addition, since $\delta\leq\delta_0$ and $\delta_0(s+1)=0$, we have $\delta(s+1)=0$.
\item To show that $J_1\in S_{A'}$ we only need to show that $|\delta|\leq n$ because we already know that $\det(L_{J_1}^c)\neq 0$. If $|\delta|>n$ then, by lemma \ref{l. basic properties} (ii), $c_{\delta}=0$, which contradicts that $\det(L_{J_1}^c)\neq 0$. 
On the other hand, we know that $\delta\leq\delta_0$ and $\gamma'\leq\beta_1$. Let $\delta_0=\delta+\delta'$ and $\beta_1=\gamma'+\gamma''$.
Then
$$A'\beta_1=A'\gamma'+A'\gamma''=A'\delta_0+A'\gamma''=A'\delta+A'\delta'+A'\gamma''.$$
This implies that $m_{\bar{J}}$ equals $m_{J_1}$ plus an element from $\Gamma$. 
\end{enumerate}
\end{proof}

%%%%%%%%%%%%%%%%%%%%%%%%%%%%%%%%%%%%
%		Higher Nash blowup toric curves
%%%%%%%%%%%%%%%%%%%%%%%%%%%%%%%%%%%%

\section{Higher Nash blowup of toric curves}\label{Nash toric curves}

In this section we study in detail the higher Nash blowup of toric curves. In this section we use the following notation: 
$A=\{a_1,\ldots,a_s\}\subset\N$, where $0<a_1<\ldots<a_s$ and $\gcd(a_1,\ldots,a_s)=1$. Let $\Gamma=\N A\subset\N$. We assume that $A$ is the minimal generating set of $\Gamma$. Let $\Xg\subset\K^s$ be the corresponding toric curve. 

According to theorem \ref{main s3}, $\Nag$ is isomorphic to the blowup of the ideal $\I_n$. Since $\Ga\subset\R_{\geq0}$, it follows that the Newton polyhedron of $\I_n$ has only one vertex $m_{J_0}=\min\{m_J|J\in S_A\}$. In particular, $\Nag$ is determined by a single semigroup. We denote it as:
$$\Nas:=\Ga+\N(\{m_J-m_{J_0}|J\in S_A\setminus\{J_0\}\}).$$

Let us show how this semigroup looks like for $n=1$. In this case, $S=\{e_1,\ldots,e_s\}$, where the $e_i's$ denote the canonical basis 
of $\N^s$, $m_{e_i}=a_i$, and so $\min_i\{m_{e_i}\}=a_1$. Therefore 
$$Nash_1(\Gamma)=\Ga+\N(\{a_k-a_1|k>1\}).$$ 

\begin{rem}
The previous description is a particular case of the combinatorial description of the Nash blowup of toric varieties given in \cite{GT, GM}
(see also \cite{DG}, where the Nash blowup of toric curves is studied in detail).
\end{rem}

We may ask the question: is there an explicit description for $\Nas$ as in $n=1$? T. Yasuda made the following conjecture in a more general context.

\begin{conj}\label{conj}\cite[Conjecture 5.6]{Y2}
Let $X$ be a formal curve with associated semigroup $\Gamma=\{0=s_0<s_1<\cdots\}$. Let $\Nas$ be the associated 
semigroup of $\Na$. Let $\Gan$ be the semigroup generated by $s_m-s_l$, where $l \leq n < m$. Then $\Nas=\Gan$.
\end{conj}

In what follows we prove that this conjecture is true for toric curves. However, in the final section we show that it fails in general.

In order to prove the conjecture in the toric case, first we need to study carefully some maximal minors of the higher-order Jacobian matrix.
In section \ref{s. monomial morphism} we defined, for 
$J=\{\beta_1,\ldots,\beta_n\}\subset\Ls$, the matrix $L_J^c=(c_{\beta_i,j})_{i,j}$, where 
$$c_{\beta_{i},j}=\sum_{\gamma\leq\beta_{i},\gamma\neq0}
(-1)^{|\beta_{i}-\gamma|}\binom{\beta_{i}}{\gamma}\binom{A\cdot\gamma}{j}.$$
Notice that in this case $A$ is a vector in $\N^s$ and $A\cdot\gamma$ is the usual dot product.

\begin{rem}\label{delta}
\begin{itemize}
\item[(a)] For a fixed $i$, every entry of the $i$-th row of $L_J^c$ has the same amount of summands and the same coefficients
$(-1)^{|\beta_{i}-\gamma|}\binom{\beta_{i}}{\gamma}$. In other words, for a fixed row of $L_J^c$, the amount of summands and 
coefficients of its entries do not depend on $j$. 
\item[(b)] Fix $i\in\{1,\ldots,n\}$. We rewrite the sums $c_{\beta_i,j}$ as follows:
$$c_{\beta_i,j}=\binom{s_{i,1}}{j}+t_{i,2}\binom{s_{i,2}}{j}+\cdots+t_{i,k_{i}}\binom{s_{i,k_{i}}}{j},$$
where $s_{i,1}:=A\cdot\beta_i$, $s_{i,l}\in\{A\cdot\gamma|\gamma\leq\beta_i,0\neq\gamma\neq\beta_i\}$ for $1<l\leq k_i$, 
and $t_{i,l}\in\Z$. Assume that $s_{i,1}>s_{i,2}>\ldots>s_{i,k_i}>0$. By (a), $\{s_{i,l}\}_l$, $\{t_{i,l}\}_l$ and $k_i$ 
do not depend on $j$. Therefore, the $i$-th row of $L_J^c$ can be written as:
\[
\small
\left(
\begin{array}{c}
\binom{s_{i,1}}{1}+t_{i,2}\binom{s_{i,2}}{1}+\cdots+t_{i,k_{i}}\binom{s_{i,k_{i}}}{1},\ldots,
\binom{s_{i,1}}{n}+t_{i,2}\binom{s_{i,2}}{n}+\cdots+t_{i,k_{i}}\binom{s_{i,k_{i}}}{n}
\end{array} 
\right).
\]
\end{itemize}
\end{rem}

Now we define some elementary operations on a matrix having the same shape as $L_{J}^{c}$. Given 
$k_i\in\N$ for $i\in\{1,\ldots,n\}$, and $s_{i,l}\in\N\setminus\{0\}$, $t_{i,l}\in\Q\setminus\{0\}$ for $l\in\{1,\ldots,k_i\}$, 
consider a matrix
$$D=
\left(
{\begin{array}{c}
t_{i,1}\binom{s_{i,1}}{j}+t_{i,2}\binom{s_{i,2}}{j}+\cdots+t_{i,k_{i}}\binom{s_{i,k_{i}}}{j}\\
\end{array}}
\right)_{\substack{1\leq i\leq n\\1\leq j\leq n}}.$$
Notice that if we fix $i$, the terms $k_{i},s_{i,l},t_{i,l}$ do not depend on $j$. Assume that
$s_{i,1}>s_{i,2}>\cdots>s_{i,k_{i}}$ 
for all $i\in\{1,\ldots,n\}$. Finally, let $R_i$ denote the $i$-th row of $D$. 

\begin{defi}
We say that $D$ satisfies $(\star)$ if there exist $i,i'\in\{1,\ldots,n\}$ such that $s_{i,1}=s_{i',1}$.
\end{defi}

Using the following algorithm we show that, under some assumptions, we can perform elementary operations on the rows of $D$
to obtain a matrix that does not satisfy the property $(\star)$.

\begin{algo}\label{algor1} Assume $\det D\neq0$ and that $D$ satisfies $(\star)$. 
\begin{enumerate}
\item Replace the row $R_{i}$ by $t_{i',1}R_{i}-t_{i,1}R_{i'}$.
\item Since $\det D\neq0$ the new row cannot be the vector $\bar{0}$. Write this new vector as:
$$R_i':=
\left( 
{\begin{array}{c}
t_{i,1}'\binom{s_{i,1}'}{j}+t_{i,2}\binom{s_{i,2}'}{j}+\cdots+t_{i,k_{i}'}'\binom{s_{i,k_{i}'}'}{j}\\
\end{array} } 
\right )_{1\leq j\leq n},$$
where $t_{i,l}'\neq0$ for all $l\in\{1,\ldots,k_{i}'\}$ and $s_{i,1}'>\cdots>s_{i,k_{i}'}'$. Notice that
$s_{i,1}>s_{i,1}'>0.$
\item Let 
$$D':=
\left( {\begin{array}{c}
R_{1}\\
\vdots\\
R_{i}'\\
\vdots\\
R_{n}\\
\end{array} } 
\right ).$$
\begin{itemize}
\item[(i)] If there exists $i''\in\{1,\ldots,n\}\setminus\{i\}$ such that $s_{i,1}'=s_{i'',1}$, then apply step 1 to $R'_i$. As before, 
we obtain a new element $s_{i,1}''\in\N$ such that $s_{i,1}>s_{i,1}'>s_{i,1}''>0$. 
\item[(ii)] If there is no $i''\in\{1,\ldots,n\}\setminus\{i\}$ such that $s_{i,1}'=s_{i'',1}$ then stop.
\end{itemize}
\end{enumerate}
Because of the decreasing sequence $s_{i,1}>s_{i,1}'>s_{i,1}''>\cdots$, this algorithm must stop and it produces a new row 
that looks like
$$\left(
{\begin{array}{c}
u_{i,1}\binom{r_{i,1}}{j}+u_{i,2}\binom{r_{i,2}}{j}+\cdots+u_{i,m_{i}}\binom{r_{i,m_{i}}}{j}\\
\end{array} } 
\right)_{1\leq j\leq n},$$
where $u_{i,l}\neq0$ for all $l\in\{1,\ldots,m_{i}\}$, $r_{i,1}>\cdots>r_{i,m_{i}}>0$ and $r_{i,1}\neq s_{l,1}$ for all 
$l\in\{1,\ldots,n\}\setminus\{i\}$.
\end{algo}

Applying this process every time that the new matrix satisfies $(\star)$, we finally get a matrix $\overline{D}$ 
$$\overline{D}=
\left(
{\begin{array}{c}
u_{i,1}\binom{r_{i,1}}{j}+u_{i,2}\binom{r_{i,2}}{j}+\cdots+u_{i,m_{i}}\binom{r_{i,m_{i}}}{j}\\
\end{array}}
\right)_{\substack{1\leq i\leq n\\1\leq j\leq n}},$$
such that $r_{i,1}>\cdots>r_{i,m_{i}}$ for each $i$ and $r_{i,1}\neq r_{i',1}$ for all $i\neq i'$.

\begin{exam}Consider the following matrix:
$$D=\left(
{\begin{array}{ccc}
\binom{2}{1}&\binom{2}{2}&\binom{2}{3}\\
\binom{6}{1}-2\binom{3}{1}&\binom{6}{2}-2\binom{3}{2}&\binom{6}{3}-2\binom{3}{3}\\
\binom{6}{1}-3\binom{4}{1}+3\binom{2}{1}&\binom{6}{2}-3\binom{4}{2}+3\binom{2}{2}&\binom{6}{3}-3\binom{4}{3}
+3\binom{2}{3}\\
\end{array}}
\right ).$$
Notice that $D$ satisfies $(\star)$. Applying algorithm \ref{algor1} to the third row we obtain the matrix:
$$\overline{D}=
\left(
{\begin{array}{ccc}
\binom{2}{1}&\binom{2}{2}&\binom{2}{3}\\
\binom{6}{1}-2\binom{3}{1}&\binom{6}{2}-2\binom{3}{2}&\binom{6}{3}-2\binom{3}{3}\\
-3\binom{4}{1}+2\binom{3}{1}+3\binom{2}{1}&-3\binom{4}{2}+2\binom{3}{2} +3\binom{2}{2}&-3\binom{4}{3}
+2\binom{3}{3}+3\binom{2}{3}\\
\end{array}}
\right ).$$
\end{exam}

%%%%%% Partial description

\subsection{A partial description of $\Nag$}\label{partial description}

The first step towards proving conjecture \ref{conj} for toric curves is to determine $\min_{J\in S_A}\{m_J\}$. Recall that for 
$J=\{\beta_1,\ldots,\beta_n\}\subset\Ls$, we defined $m_J=A\cdot\beta_1+\cdots+A\cdot\beta_n$. On the 
other hand, $A\cdot\beta_i\in\Gamma$ since $A$ is the vector formed by the generators of $\Gamma$.
Therefore, it is natural to expect that $\min_{J\in S_A}\{m_J\}=\sum_{i=1}^ns_i$. The goal of this subsection is to
prove that this is indeed the case.

\begin{pro}\label{minimo J}
Let $J\subset\Ls$, $|J|=n$, then $\min_{J\in S_A}\{m_{J}\}=\sum_{i=1}^{n}s_{i}$.
\end{pro}
\begin{proof}
This is proved in lemmas \ref{l. zero} and \ref{u. bound}.
\end{proof}

This proposition gives the following partial description of $\Nas$.
\begin{coro}\label{str. nash}
$Nash_{n}(\Ga)=\Ga+\N(\{m_{J}-\sum_{i=1}^{n}s_{i}|J\in S_A\}).$
\end{coro}

\begin{lem}\label{l. zero} Let $J\in S_A$. Then $m_J\geq s_1+\cdots+s_n$. 
In particular, $\min_{J\in S_A}\{m_J\}\geq\sum_{i=1}^ns_i$.
\end{lem}

\begin{proof}
Let $J=\{\be_{1},\be_{2},\ldots,\be_{n}\}$. Using (b) of remark \ref{delta} we have
$$L_{J}^{c}=
\left(
{\begin{array}{c}
\binom{s_{i,1}}{j}+t_{i,2}\binom{s_{i,2}}{j}+\cdots+t_{i,k_{i}}\binom{s_{i,k_{i}}}{j}\\
\end{array}}
\right)_{\substack{1\leq i\leq n\\1\leq j\leq n}},$$
where $s_{i,l}\in\Ga$ for each $l$, $s_{i,1}=A\cdot \beta_i$, and $s_{i,1}>s_{i,2}>\cdots>s_{i,k_i}$. 
If $s_{i,1}\neq s_{i',1}$ for all $1\leq i\neq i'\leq n$ then the statement follows.

Now suppose that there exist $\be_{i},\be_{i'}\in J$, $i\neq i'$, such that $s_{i,1}=s_{i',1}$, i.e., 
$L_{J}^{c}$ satisfies $(\star)$. Since $J\in S_A$, i.e., $\det(L_J^c)\neq0$, we can apply algorithm \ref{algor1} 
to obtain some elements $r_{1,1},\ldots,r_{n,1}\in\Ga$ satisfying $r_{i,1}\neq r_{i',1}$ for all $i\neq i'$ and $s_{i,1}>r_{i,1}$ for 
some $i\in\{1,\ldots,n\}$. Under these conditions we have
$$m_{J}=\sum_{i=1}^{n} A\cdot\beta_i=\sum_{i=1}^{n}s_{i,1}>\sum_{i=1}^{n}r_{i,1}\geq\sum_{i=1}^{n}s_{i}.$$
\end{proof}

To show that $\min_{J\in S_A}\{m_{J}\}=\sum_{i=1}^{n}s_{i}$ we need to show that, if $J=\{\beta_1,\ldots,\beta_n\}\subset\Ls$ is
such that $A\cdot\beta_i=s_i$, then $J\in S_A$. In other words, we need to study $J$'s such that $\det(L_J^c)\neq0$. We do not have
a characterization of such $J$'s. However, in the following definition and lemma we give sufficient conditions for $J$ to be in $S_A$.

\begin{defi}\label{aster}
Let $J\subset\N^s$ be a finite subset. We say that $J$ satisfies $(\ast)$ if the following conditions hold:
\begin{itemize}
\item[1)]For all $\beta,\beta'\in J$ such that $\beta\neq\beta'$, it holds $A\cdot\beta\neq A\cdot\beta'$.
\item[2)] For all $\beta\in J$ and $0\neq\gamma<\beta$, there exists $\beta'\in J$ such that $A\cdot\gamma=A\cdot\beta'$.
\end{itemize}
\end{defi}

\begin{exam}\label{exam aster}
\begin{itemize}
\item[(i)] Let $J=\{e_{j},2e_{j},\ldots,ne_{j}\}$, where $e_{j}$ is a basic vector. Then $J$ satisfies $(\ast)$. Indeed, 1) follows 
by definition and 2) by the definition of $<$.
\item[(ii)] Let $J=\{\be_{1},\be_{2},\ldots,\be_{n}\}$ be such that $A\cdot\be_{i}=s_{i}$. Then $J$ satisfies $(\ast)$. Indeed,
1) follows by definition and 2) follows from the fact that $\gamma<\beta$ implies $A\cdot\gamma<A\cdot\beta$.
\end{itemize}
\end{exam}

\begin{rem}\label{j0}
Let $\be\in\N^{s}$ be such that $|\be|\geq n+1$. Then $A\cdot\be> na_{1}\geq s_{n}$. This implies that
for $J=\{\be_{1},\ldots,\be_{n}\}\subset\N^{s}$ such that $A\cdot\be_{i}=s_{i}$ for each $i$, we have
$J\subset\Lambda_{s,n}$.
\end{rem}

\begin{lem}\label{u. bound}
Let $J\subset \Lambda_{s,n}$, $|J|=n$. If $J$ satisfies $(\ast)$ then $J\in S_A$. 
In particular, $\min_{J\in S_A}\{m_J\}\leq\sum_{i=1}^ns_i$.
\end{lem}
\begin{proof}
The last statement follows from (ii) of example \ref{exam aster} and remark \ref{j0}. Let $J=\{\beta_1,\ldots,\beta_n\}$ satisfies 
$(\ast)$. In particular, $A\cdot\beta_i\neq A\cdot\beta_{i'}$ for all $i\neq i'$. Assume $A\cdot\beta_1<\cdots<A\cdot\beta_n$.  
By (b) of remark \ref{delta}, we can rewrite the matrix $L_{J}^{c}$ as
$$\left(
{\begin{array}{c}
c_{\be_{i},j} \\
\end{array}}
\right)_{\substack{1\leq i\leq n\\1\leq j\leq n}}
=
\left(
{\begin{array}{c}
\binom{s_{i,1}}{j}+t_{i,2}\binom{s_{i,2}}{j}+\cdots+t_{i,k_{i}}\binom{s_{i,k_{i}}}{j}\\
\end{array} }
\right)_{\substack{1\leq i\leq n\\1\leq j\leq n}},$$
where $s_{i,1}=A\cdot\be_{i}$, $s_{i,1}>s_{i,2}>\cdots$, and $s_{i,l}=A\cdot \ga$ for some $\ga\leq\be_{i}$.

Now we do some elementary operations on the $n$-th row of $L_J^c$:
$$\left( {\begin{array}{c}
 \binom{s_{n,1}}{j}+t_{n,2}\binom{s_{n,2}}{j}+\cdots+t_{n,k_{n}}\binom{s_{n,k_{n}}}{j} \\
\end{array} } \right )_{1\leq j\leq n}.$$
We know that $s_{n,2}=A\cdot\ga$ for some $\ga<\be_{n}$. Since $J$ satisfies $(\ast)$, there exists $\be_{j_0}\in J$ such that 
$A\cdot\be_{j_0}=A\cdot\ga=s_{n,2}$. Then we sustract $t_{n,2}$-times the row $j_0$ to the row $n$, thus obtaining
$$\left( {\begin{array}{c}
 \binom{s_{n,1}'}{j}+t_{n,2}'\binom{s_{n,2}'}{j}+\cdots+t_{n,k_{n}'}'\binom{s_{n,k_{n}'}'}{j} \\
\end{array} } \right )_{1\leq j\leq n},$$
where $s'_{n,1}=s_{n,1}$, $s'_{n,2}>s'_{n,3}>\cdots$, and $s_{n,2}>s_{n,2}'$. Now we have $s_{n,2}'=A\cdot\ga'$ for some 
$\ga'<\be_{n}$ or some $\ga'<\be_{j_0}$. Once again, by $(\ast)$ we can repeat the previous process to obtain a new element
$s''_{n,2}$ such that $s_{n,2}>s'_{n,2}>s''_{n,2}$. Because of this decreasing sequence of natural numbers, the iteration of this 
process must stop turning the $n$-th row into
$$\left(
{\begin{array}{c}
 \binom{s_{n,1}}{j} \\
\end{array}}
\right )_{1\leq j\leq n}.$$

Applying this process to the other rows of $L_J^c$ in an ascending way we obtain the matrix
$$\left(
{\begin{array}{c}
 \binom{s_{i,1}}{j} \\
\end{array}}
\right )_{\substack{1\leq i\leq n\\1\leq j\leq n}}.$$
Notice that $s_{i,1}=A\cdot\beta_i\neq A\cdot\beta_{i'}=s_{i',1}$ for all $i\neq i'$. The following lemma 
shows that this matrix has non-zero determinant, thus concluding that $J\in S$.
\end{proof}

\begin{lem}\label{lema matriz}
Let $0<c_{1}<c_{2}<\cdots<c_{n}$ be natural numbers. Consider the matrix
$L=\left(
{\begin{array}{c}
\binom{c_{i}}{j}\\
\end{array}}
\right )_{\substack{1\leq i\leq n\\1\leq j\leq n}}.$
Then $\det L\neq0$.
\end{lem}
\begin{proof}
For $j\in\{1,\ldots,n\}$, consider the polynomial $b_{j}(x)=\frac{x(x-1)\cdots(x-j+1)}{j!}$. Notice that if $x\in\N$, 
$b_{j}(x)=\binom{x}{j}$ and $\deg{b_{j}}(x)=j$. Thus 
$$L=
\left(
{\begin{array}{c}
\binom{c_{i}}{j}\\
\end{array}}
\right )_{\substack{1\leq i\leq n\\1\leq j\leq n}}
=\left(
{\begin{array}{c}
b_{j}(c_{i})\\
\end{array}}
\right )_{\substack{1\leq i\leq n\\1\leq j\leq n}}.$$
Now we show that the columns of this matrix are linearly independent. Let $\al_{1},\ldots,\al_{n}\in\R$ be such that 
$\sum_{j=1}^{n}\al_{j}b_{j}(c_{i})=0$, for each $i\in\{1,\ldots,n\}$. Let $f(x)=\sum_{j=1}^{n}\al_{j}b_{j}(x)$. Then
$\{c_1,\ldots,c_n\}$ are roots of $f(x)$. But we also have $f(0)=0$. Since $\deg f(x)\leq n$ it follows that $f(x)=0$. As 
$\deg{b_{j}(x)}=j$ for each $j$, we conclude that $\al_{j}=0$ for all $j$. In particular, $\det L\neq 0$.
\end{proof}

\subsection{Proof of conjecture \ref{conj} for toric curves and some consequences}

Now we are ready to prove the main theorem of this section. Recall that by definition and corollary \ref{str. nash}:
\begin{align}
\Ga^{(n)}&=\N(\{s_m-s_l|m>n,l\leq n\}),\notag\\
\Nas&=\Ga+\N(\{m_{J}-\sum_{l=1}^{n}s_l|J\in S_A\}).\notag
\end{align}

\begin{teo}\label{main t}
$\Ga^{(n)}=Nash_{n}(\Ga)$.
\end{teo}
\begin{proof}
This is proved in propositions \ref{contencion} and \ref{equality sem}.
\end{proof}

\begin{pro}\label{contencion}$Nash_{n}(\Ga)\subset \Ga^{(n)}$.
\end{pro}
\begin{proof}
By corollary \ref{str. nash}, it is enough to show that $a_i\in\Ga^{(n)}$ for each $i\in\{1,\ldots,s\}$ and $m_J-\sum_{l=1}^ns_l\in\Ga^{(n)}$ for each $J\in S_A$.

We first prove $a_i\in\Ga^{(n)}$. For $a_{i}\leq s_{n}$ there exists $m\in\N$ such that $ma_{i}\leq s_{n}<(m+1)a_{i}$. 
Then $a_i=(m+1)a_{i}-ma_i\in\Ga^{(n)}$. If $a_{i}\geq s_{n}$ then $a_{i}+a_{1}>s_{n}$, and 
$a_{i}=(a_{i}+a_{1})-a_{1}\in\Ga^{(n)}$.

Now we prove that $m_{J}-\sum_{l=1}^{n}s_{l}\in\Ga^{(n)}$ for each $J\in S_A$. Consider 
$J=\{\be_{1},\be_{2},\ldots,\be_{n}\}\in S_A$, 
let $s_{i,1}:=A\cdot\be_{i}$ and assume $s_{1,1}\leq\cdots\leq s_{n,1}$. Let $k:=\max\{l\in\{1,\ldots,n\}|s_{l,1}\leq s_{n}\}$. 

\noindent\textbf{Case I:} Suppose that $s_{1,1}<s_{2,1}<\cdots<s_{k,1}\leq s_n$. Let $\psi=\{s_1,\ldots,s_n\}\setminus\{s_{1,1},\cdots,s_{k,1}\}$.
Write $\psi=\{r_{k+1},\ldots,r_n\}$. By definition of $k$ and $\psi$ we obtain:
$$m_J-\sum_{l=1}^{n}s_{l}=\sum_{l=1}^{n}s_{l,1}-\sum_{l=1}^{n}s_{l}=\sum_{l=k+1}^{n}s_{l,1}-\sum_{l=k+1}^n r_{l}\in\Gan.$$

\noindent\textbf{Case II: } Suppose that there exist $i,i'\leq k$ such that $s_{i,1}=s_{i',1}$. We claim that for all $j\leq k$ there exist $r_{j,1}\in\Ga$ such that $s_{j,1}-r_{j,1}\in\Ga^{(n)}$ and $r_{j,1}\neq r_{j',1}$ for all $j\neq j'$. Assume this claim for the moment. 
For the elements $s_{m,1}$ with $m>k$, we have that $s_{m,1}>s_n$ and so $s_{m,1}-s_{l}\in\Ga^{(n)}$ for any $l\leq n$. 
Let $\psi=\{s_1,\ldots,s_n\}\setminus\{r_{1,1},\ldots,r_{k,1}\}$. Write $\psi=\{r_{k+1},\ldots,r_n\}$. As in the previous case, we conclude that
$$m_J-\sum_{l=1}^{n}s_{l}=\Big(\sum_{l=1}^{k}s_{l,1}-\sum_{l=1}^{k}r_{l,1}\Big)
+\Big(\sum_{l=k+1}^{n}s_{l,1}-\sum_{l=k+1}^n r_{l}\Big)\in\Gan.$$

Now we prove the claim. Since $J\in S_A$, we can apply algorithm \ref{algor1} to any pair of rows of 
$L_J^c$, $i,i'\leq k$ such that $s_{i,1}=s_{i',1}$, to get a matrix 
$$\overline{D}=
\left(
{\begin{array}{c}
u_{i,1}\binom{r_{i,1}}{j}+u_{i,2}\binom{r_{i,2}}{j}+\cdots+u_{i,m_{i}}\binom{r_{i,m_{i}}}{j}\\
\end{array}}
\right )_{\substack{1\leq i\leq n\\1\leq j\leq n}},$$
where $r_{i,1}\neq r_{i',1}$ for all $i,i'\leq k$. Let us show that $s_{i,1}-r_{i,1}\in\Ga^{(n)}$ for all $i\in\{1,\ldots,k\}$.
We can assume that $s_{i,1}\neq r_{i,1}$.

In the first run of the algorithm we obtain an element $s'_{i,1}\in\Gamma$ such that $s_{i,1}>s_{i,1}'\geq r_{i,1}$ and
$s'_{i,1}=A\cdot\ga$ for some $\ga<\be_{i}$ or some $\ga<\be_{i'}$. This implies that $s_{i,1}-s_{i,1}'\in\Ga$.

On the other hand, we know that $s_{i,1}'\geq r_{i,1}$ and $r_{i,1}\in\Ga$. Therefore $s_{i,1}-s'_{i,1}+r_{i,1}\in\Ga$ and 
$s_{i,1}-s'_{i,1}+r_{i,1}\leq s_{n}$. Consider the following set $\phi_{i}:=\{s_{l}\in\Ga\setminus\{0\}| s_{l}+r_{i,1}\leq s_{n}\}$. 
This set is not empty since $s_{i,1}-s_{i,1}'\in\phi_{i}$. Let $s_{t}:=\max\phi_{i}$. If $s_{i,1}+s_{t}\leq s_{n}$, we have 
that $(s_{i,1}-s_{i,1}'+s_{t})+r_{i,1}=s_{i,1}+s_t-(s_{i,1}'-r_{i,1})\leq s_{i,1}+s_{t}\leq s_{n}$ and $s_{i,1}-s_{i,1}'+s_{t}>s_{t}$, 
which contradicts the maximality of $s_{t}$. Thus $s_{i,1}+s_t>s_{n}$ and
$$s_{i,1}-r_{i,1}=(s_{i,1}+s_t)-(s_{t}+r_{i,1})\in\Ga^{(n)}.$$
\end{proof}

We need the following lemma to prove the remaining inclusion in theorem \ref{main t}.

\begin{lem}\label{beta} 
Let $s_{m},s_{i}\in\Ga$ be such that $m>n\geq i$ and $s_{m}-s_{i}\notin\Ga$. Let $\beta_{m}\in\N^{s}$ be such that
$A\cdot\beta_{m}=s_{m}$. Then there exists $\beta_0\leq\beta_{m}$ such that $A\cdot\beta_0>s_{n}$ and $|\beta_0|\leq n$.
\end{lem}
\begin{proof}
If $|\be_{m}|\leq n$ then $\be_{m}$ satisfies the conditions of $\be_0$. Assume that $|\be_{m}|>n$.

Suppose first that $a_{2}\leq s_{n}$. The set $\{a_{1},2a_{1}, \ldots,(n-1)a_{1},a_{2}\}$ has $n$ different elements of 
$\Gamma$ implying $na_{1}>s_{n}$. Let $\be$ be such that $|\be|=n$. Then $A\cdot\be\geq ns_{1}=na_{1}>s_{n}$. 
In particular, any $\be\leq\be_{m}$ such that $|\be|=n$ satisfies the conditions on the lemma.

Now suppose $s_{n}<a_{2}$. Then $s_{k}=ks_{1}$ for all $k\leq n$. Notice that if $\be_{m}(j)=0$ for all $j>1$, then 
$s_{m}-s_{i}=|\be_{m}|a_{1}-ia_{1}\in\Ga$ which contradicts the hypothesis. Thus there exists $j>1$ such that 
$\be_{m}(j)\neq 0$. Consider $\be_0=e_{j}$, then $A\cdot\be_0=A\cdot e_{j}=a_{j}\geq a_{2}>s_{n}$ and $|\be_0|=1\leq n$.
\end{proof}

\begin{pro}\label{equality sem}$\Ga^{(n)}\subset Nash_{n}(\Ga)$.
\end{pro}
\begin{proof}
Throughtout this proof we fix $m,i\in\N$ such that $m>n\geq i$. Let $s_{m}-s_{i}\in\Ga^{(n)}$. 

\noindent\textbf{Case I:} Suppose that $s_{m}-s_{i}\in\Ga$. Then $s_{m}-s_{i}\in\Nas$ by definition.

\noindent\textbf{Case II:} Suppose that $s_{m}-s_{i}\notin\Ga$. Fix $\be_m\in\N^s$ such that $A\cdot \be_m=s_m$. 

We claim that there exist $\be_0\in\N^s$, $J_0=\{\be_1,\ldots,\be_n\}\subset\N^s$, and $i\leq l \leq n$ such that:
\begin{itemize}
\item[(1)] $\beta_0\leq\beta_m$.
\item[(2)] $A\cdot\be_j=s_j$ for each $j\in\{1,\ldots,n\}$.
\item[(3)] $s_l-s_i\in\Ga$.
\item[(4)] $J:=(J_{0}\setminus\{\be_{l}\})\cup\{\be_0\}$ satisfies $(\ast)$. In particular, 
$$A\cdot\be_0-s_l=m_J-\sum_{i=1}^ns_i\in Nash_n(\Ga).$$
\end{itemize}
Assume this claim for the moment. Let $\delta\in\N^s$ be such that $\be_m=\be_0+\delta$. In particular, $s_m=A\cdot{\be_m}=A\cdot{\be_0}+A\cdot{\delta}$. Then, since $A\cdot\delta\in\Ga$, we conclude
$$s_{m}-s_{i}=(A\cdot\be_0-s_{l})+A\cdot\delta+(s_{l}-s_{i})\in Nash_{n}(\Ga).$$

Now we prove the claim. We first show that there is a $\be_0\in\N^s$ satisfying (1) and some extra conditions needed for the proof of (4). Let $T:=\{\ga\in\N^{s}|\ga\leq\be_{m}\}$. We write this set as $T=T_{\leq}\sqcup T_{>}$, where 
\begin{align}
T_{\leq}&:=\{\ga\in T|A\cdot\ga\leq s_{n}\},\notag\\
T_{>}&:=\{\ga\in T|A\cdot\ga>s_{n}\}.\notag
\end{align}
Notice that $\be_{m}\in T_{>}$. Let $\be_0\leq\be_{m}$ be a minimal element in $T_{>}$ such that $\be_0\in\Lambda_{s,n}$ 
(such an element exists by lemma \ref{beta}). By construction, $\be_0$ has the following properties:
\begin{itemize}\item[a)]For all $\ga<\be_0$, $\ga\in T_{\leq}$.
\item[b)]For all $\bar{\be_i}$ such that $A\cdot\bar{\be_i}=s_{i}$ it holds $\be_0\ngtr\bar{\be_i}$ (this is true because 
$\beta_m\geq\beta_0$ and $s_{m}-s_{i}\notin\Ga$ implies that $\be_{m}\ngtr\bar{\be_{i}}$).
\end{itemize} 

Now we prove the existence of $J_0=\{\be_1,\ldots,\be_n\}\subset\N^s$, and $i\leq l \leq n$ satisfying (2) and (3) and some extra conditions needed for the proof of (4). Define the set (recall that $i$ is fixed):
$$\Omega:=\{s_{j}\in\{s_{1},\ldots,s_{n}\}|\forall\be'\mbox{ such that }A\cdot\be'=s_{j},\exists
\ga\leq\be'\mbox{ such that }A\cdot\ga=s_{i}\}.$$
If $s_{j}\in\{s_{1},\ldots,s_{n}\}\setminus\Omega$, consider $\be_{j}\in\N^{s}$ such that $A\cdot\be_{j}=s_{j}$ and for all 
$\ga\leq\be_{j}$, $A\cdot\ga\neq s_{i}$. If $s_{j}\in\Omega$, consider any $\be_{j}\in\N^{s}$ such that $A\cdot\be_{j}=s_{j}$. 
Let $J_{0}:=\{\be_{1},\ldots,\be_{n}\}\subset\N^{s}$. By remark \ref{j0}, we have that $J_{0}\subset\Lambda_{s,n}$.
Notice that $\Omega\neq\emptyset$, since $s_{i}\in\Omega$. Let $s_{l}:=\max\{\Omega\}$. In particular, $s_l\in\Omega$ and so $s_{l}-s_{i}\in\Ga$.

It remains to prove that $J:=(J_{0}\setminus\{\be_{l}\})\cup\{\be_0\}$ satisfies $(\ast)$ (see definition \ref{aster}). By construction and since $A\cdot\be_0> s_{n}$, we have 1) in the definition of $(\ast)$. 

Now let $\be_{k}\in J_{0}\setminus\{\be_{l}\}$. We want to show that if $\gamma<\beta_k$ then
there exists $\beta'\in J$ such that $A\cdot\beta'=A\cdot\gamma$. If $k<l$ this condition is satisfied (see example \ref{exam aster}). 
Suppose $k>l$ (in particular, $s_k\notin\Omega$). If $\ga<\be_{k}$ is such that $A\cdot\ga=s_j\neq s_{l}$ then by 
making $\beta'=\beta_j$ the condition is satisfied. Suppose that  $A\cdot\ga=s_{l}$. Since $s_{l}\in\Omega$ there exists 
$\ga'\leq\ga$ such that $A\cdot\ga'=s_{i}$. Since $\ga<\be_{k}$, it follows that $\ga'<\be_{k}$. This is a contradiction since
$\beta_k$ was chosen so that for all $\delta<\beta_k$ we have $A\cdot\delta\neq s_i$. Therefore, for all  $\ga\leq\be_{k}$,
$A\cdot\ga\neq s_{l}$. This shows that every element of $J_{0}\setminus\{\be_{l}\}$ satisfies 2) in the definition of $(\ast)$. 

Now consider $\ga<\be_0$. By property a) above, we have that $A\cdot\ga\leq s_{n}$. 
If $\ga<\be_0$ is such that $A\cdot\ga=s_j\neq s_{l}$ then,
as before, by making $\beta'=\beta_j$ the condition is satisfied. Suppose that  $A\cdot\ga=s_{l}$. As before, there exists 
$\ga'\leq\ga$ such that $A\cdot\ga'=s_{i}$. Since $\ga<\be_0$, it follows that $\ga'<\be_0$. This contradicts property b) above.
We conclude that $J$ satisfies $(\ast)$. 
\end{proof}

Theorem \ref{main t} has two immediate consequences. The first one is about solving toric curves by applying once the higher 
Nash blowup for $n$ sufficiently large. This gives a combinatorial proof of Yasuda's theorem on one-step resolution of curves by 
higher Nash blowups in the case of toric curves. The second result is the analogous of Nobile's theorem for the higher Nash blowup 
of toric curves.

\begin{coro}\label{yasuda teo}
$Nash_{n}(X_{\Ga})$ is non-singular if and only if $s_{n}+1\in\Ga$. In particular, $Nash_{n}(X_{\Ga})$ is non-singular for $n\gg 0$.
\end{coro}
\begin{proof}
Notice that for all $m>n$ and $i\leq n$, we have that $s_{n+1}+s_{i}\leq s_{m}+s_{n}$, then $s_{n+1}-s_{n}\leq s_{m}-s_{i}$. 
Thus, $s_{n+1}-s_{n}=\min \{\Ga^{(n)}\setminus\{0\}\}=\min \{Nash_{n}(\Ga)\setminus\{0\}\}$. Then $Nash_{n}(X_{\Ga})$ 
is non-singular if and only if $Nash_{n}(\Ga)=\N(\{1\})$ if and only if $1=s_{n+1}-s_{n}$ if and only if $s_{n}+1=s_{n+1}\in\Ga$. 
\end{proof}

\begin{coro}\label{nobile}
$\Nag\cong\Xg$ if and only if $\Xg$ is non-singular.
\end{coro}
\begin{proof}
Suppose that $\Xg$ is singular, i.e., $1<a_{1}$. We are going to show that $\Gamma\subsetneq\Nas$, which implies 
$\Nag\not\cong\Xg$.

Let $a_{2}=qa_{1}+r$, where $0<r<a_{1}$ and $q\geq1$. With this notation we have $s_1=a_1,\ldots,s_q=qa_1,s_{q+1}=a_2$.
If $n\leq q$ then $s_{n}\leq s_{q}=qa_{1}<a_{2}$ and so $a_2-a_1\in\Gan=\Nas$. But we also have 
$a_{2}-a_{1}=(q-1)a_{1}+r\notin\Ga$. 

Suppose that $n>q$. Consider the following subset of $\Gamma$:
$$\{s_{q+1}=qa_{1}+r,(q+1)a_{1},(q+1)a_{1}+r,(q+2)a_{1},(q+2)a_{1}+r,\ldots\}.$$
The elements on this subset are not necessarily consecutive elements in $\Gamma$. Therefore, for $p>q$ 
it follows $s_{p+1}-s_{p}\leq\max\{a_{1}-r,r\}<a_{1}$. In particular, $s_{n+1}-s_{n}<a_{1}$. Thus, 
$s_{n+1}-s_{n}\in Nash_{n}(\Ga)$ but $s_{n+1}-s_{n}\notin\Ga$.
\end{proof}

%%%%%%%%%%%%%%%%%%%%%%%%%%%%%%%%%%%%
%		Counterexample
%%%%%%%%%%%%%%%%%%%%%%%%%%%%%%%%%%%%

\section{Counterexample to the conjecture}

In section \ref{Nash toric curves} we stated and proved a conjecture by T. Yasuda for toric curves. In this section we exhibit a family of non-monomial curves showing that the conjecture is false in general. 

\begin{exam}\label{counter}
Consider the plane curve $C\subset\C^2$ parametrized by 
$$t\mapsto(t^{4},t^{6}+t^{7}).$$ 
The associated semigroup of $C$ is $\Ga=\{0,4,6,8,10,12,13,14,m|m\geq 16\}$. Yasuda's conjecture states that the semigroup
of $Nash_1(C)$ is $\Gamma^{(1)}=\N(2,9)$.
However, the Nash blowup of order 1 of $C$ is parametrized by 
$$t\mapsto\Big(t^{4},t^{6}+t^{7},\frac{6}{4}t^{2}+\frac{7}{4}t^{3}\Big).$$
Using the first and third terms of the parametrization we obtain $Nash_1(\Gamma)=\N(2,5)$. We conclude that 
$Nash_1(\Gamma)\neq\Gamma^{(1)}$.
\end{exam}

We may still ask whether the conjecture holds for $n\gg0$. In what follows we construct a family of plane curves 
$\{C_n\}_{n\geq1}$, with numerical semigroup $\Ga_n$, such that $Nash_n(\Ga_n)\neq(\Ga_n)^{(n)}$.

Fix $n\geq1$. Consider the plane curve $C_n$ with parametrization
$$\varphi(t)=(t^{4},t^{4n+2}+t^{4n+3}).$$ 
Let $\Gamma_n$ be the corresponding semigroup. Notice that the first $n$ non-zero terms of $\Gamma_n$ is the set 
$\{4,8,\ldots,4n\}$. In addition, the first odd number that appears in $\Gamma_n$ is $8n+5$ (it appears as the order of 
$(t^{4n+2}+t^{4n+3})^2-(t^4)^{2n+1}$). In particular, the first odd number that appears in $(\Ga_n)^{(n)}$ is 
$8n+5-4n=4n+5$. We claim that $5\in Nash_n(\Gamma_n)$ implying that $Nash_n(\Ga_n)\neq(\Ga_n)^{(n)}$.

To prove the claim we need to compute some maximal minors of the matrix
$$\Big(\frac{1}{\alpha!}\frac{\partial^{\alpha}(\vp-\vp(t))^{\beta}}{\partial T^{\alpha}}|_t\Big)
_{\beta\in\Lambda_{2,n}, \alpha\in\Lambda_{1,n}}$$
Let $J_1=\{e_1,2e_1,\ldots,ne_1\}$ and $J_2=\{e_1,2e_1\ldots,(n-1)e_1,e_2\}$. 
We first show that the minors of the submatrices defined by $J_{1}$ and $J_{2}$ are not zero.

Let $L_{J_1}$ be the submatrix defined by $J_1$. Notice that the rows of $L_{J_1}$ only consider the first term of $\varphi(t)$,
which is a monomial. Therefore, by example \ref{exam aster} and lemma \ref{u. bound}, $\det L_{J_{1}}\neq0$. In addition, 
by proposition \ref{p. minors}, $\det L_{J_{1}}=c\cdot t^{\sum_{k=1}^{n}4k-k}=c\cdot t^{\frac{3n(n+1)}{2}}$, with $c$ a 
non-zero constant.

Now, for $J_2$, notice that the first $n-1$ rows of $L_{J_2}$ only consider the monomial term of $\varphi(t)$. Using 
lemma \ref{coeficientes} we obtain that the $(i,j)$-entry of $L_{J_2}$ is $c_{ie_1,j}t^{4i-j}$, for $1\leq i < n$ and $1 \leq j \leq n$. 
On the other hand, the $nth$ row of $L_{J_2}$ can be described as follows. Since $|e_{2}|=1$, 
by lemma \ref{l. basic properties} we obtain 
$$\frac{1}{j !}\frac{\partial^{j}(\vp-\vp(t))^{e_{2}}}{\partial T^{j}}|_{t}=\binom{4n+2}{j}t^{4n+2-j}+\binom{4n+3}{j}t^{4n+3-j}.$$
Summarizing, the matrix $L_{J_2}$ is:
$$\left( \begin{array} {ccc}
c_{e_{1},1}t^{4-1} & \cdots & c_{e_{1},n}t^{4-n} \\
\vdots & & \vdots \\
c_{(n-1)e_{1},1}t^{4(n-1)-1} & \cdots & c_{(n-1)e_{1},n}t^{4(n-1)-n} \\
\binom{4n+2}{1}t^{4n+2-1}+\binom{4n+3}{1}t^{4n+3-1} & \cdots & \binom{4n+2}{n}t^{4n+2-n}+
\binom{4n+3}{n}t^{4n+3-n}
\end{array} \right).$$
Multiply the $jth$ column by $t^j$. Then, for $1 \leq i <n$ multiply the $ith$ row by $t^{-4i}$. Finally, multiply the $nth$ row 
by $t^{-4n-2}$ to obtain
$$\det L_{J_{2}}=\big(t^{\frac{3n(n+1)}{2}+2}\big)
\det\left( \begin{array} {ccc}
c_{e_{1},1} & \cdots & c_{e_{1},n}\\
\vdots & & \vdots \\
c_{(n-1)e_{1},1} & \cdots & c_{(n-1)e_{1},n} \\
\binom{4n+2}{1}+\binom{4n+3}{1}t& \cdots & \binom{4n+2}{n}+\binom{4n+3}{n}t
\end{array} \right).$$
Applying the same method of the proof of proposition \ref{u. bound} in the first $n-1$ rows and using basic 
properties of determinants, we get that
\begin{align*}
\det L_{J_{2}} & = & t^{\frac{3n(n+1)}{2}+2}\det\left( \begin{array} {ccc}
\binom{4}{1} & \cdots & \binom{4}{n}\\
\vdots & & \vdots \\
\binom{4(n-1)}{1} & \cdots & \binom{4(n-1)}{n} \\
\binom{4n+2}{1} & \cdots & \binom{4n+2}{n}
\end{array} \right) \\
 & + & t^{\frac{3n(n+1)}{2}+3}\det\left( \begin{array} {ccc}
\binom{4}{1} & \cdots & \binom{4}{n}\\
\vdots & & \vdots \\
\binom{4(n-1)}{1} & \cdots & \binom{4(n-1)}{n} \\
\binom{4n+3}{1}& \cdots & \binom{4n+3}{n}
\end{array} \right).
\end{align*}
The determinants appearing in the sum are non-zero by lemma \ref{lema matriz}. Therefore $\det L_{J_2}\neq0$.

Now we need to prove that $\det L_{J_{1}}$ has the minimum order over all non-zero minors of the higher-order Jacobian 
matrix of $\varphi$.

Consider $\be=(a_{1},a_{2})\in\N^{2}$ and suppose that $a_{2}>0$. Notice that if the polynomial 
$$\frac{1}{m!}\frac{\partial^{m}(T^{4}-t^{4})^{a_{1}}(T^{4n+2}+T^{4n+3}-t^{4n+2}-t^{4n+3})^{a_{2}}}
{\partial T^{m}}\big|_{t}$$
is non-zero, then its order is greater or equal than $4n+2-m$. Let $J=\{\be_{1},\ldots,\be_{n}\}\subset\Lambda_{2,n}$ 
be such that $J\neq\{e_{1},\dots,ne_{1}\}$. In particular, the second entry of $\be_i$ is non-zero, for some $i$. 
Reorder $J$ in such a way that $\be_{i}(2)\neq0$ for $1\leq i\leq k$ and $\be_{j}(2)=0$ for $k<j\leq n$. Then, 
if $j>k$, $\be_{j}=m_{j}e_{1}$ with $1\leq m_{j}\leq n$ and if $j>i>k$, $m_{j}\neq m_{i}$.

Let us show that if $\det L_{J}\neq0$ then $ord(\det(L_J))>\frac{3n(n+1)}{2}$. To begin with,
$$\det L_{J}=\sum_{\sigma\in S_{n}}sgn(\sigma)\prod_{i=1}^{n}a_{i,\sigma(i)}=\sum_{\sigma\in S_{n}}
sgn(\sigma)\prod_{i=1}^{n}\frac{1}{\sigma(i)!}\frac{\partial^{\sigma(i)}(\varphi-\varphi(t))^{\be_{i}}}{\partial T^{\sigma(i)}}|_{t},$$
where $S_{n}$ is the symmetric group. Let 
$A_{\sigma}=\prod_{i=1}^{n}\frac{1}{\sigma(i)!}\frac{\partial^{\sigma(i)}(\varphi-\varphi(t))^{\be_{i}}}{\partial T^{\sigma(i)}}|_{t}$. 
The claim follows if we can prove that, for all $\sigma$ such that $A_{\sigma}\neq 0$, $ord(A_{\sigma})>\frac{3n(n+1)}{2}$. 
But this is true since:
\begin{align}
ord(A_{\sigma})&=\sum_{i=1}^{n}ord\Big(\frac{\partial^{\sigma(i)}(\varphi-\varphi(t))^{\be_{i}}}{\partial T^{\sigma(i)}}|_{t}\Big)\notag\\
&\geq \sum_{i=1}^{k}(4n+2-\sigma(i))+\sum_{j=k+1}^{n}(4m_{j}-\sigma(j))\notag\\
&=2k+4(nk+\sum_{j=k+1}^{n}m_{j})-\frac{n(n+1)}{2}\geq 2k+4\sum_{j=1}^{n}j-\frac{n(n+1)}{2}\notag\\
&=2k+\frac{3n(n+1)}{2}>\frac{3n(n+1)}{2}.\notag
\end{align}

Using all previous claims we see that the Pl\"ucker coordinates of $T^n_{\varphi(t)}C_n$, for $t\neq0$, look like:
$$(\cdots:ct^{\frac{3n(n+1)}{2}}:c_1t^{\frac{3n(n+1)}{2}+2}+c_2t^{\frac{3n(n+1)}{2}+3}:\cdots),$$ 
with $c,c_{1},c_{2}$ non-zero constants. Since the coordinate defined by $J_{1}$ has the minimum order, 
$Nash_{n}(C_{n})\subset U_{J_{1}}$, i.e., the affine chart obtained from dividing over $ct^{\frac{3n(n+1)}{2}}$. 
In particular, the parametrization of $Nash_{n}(C_{n})$ has the term
$$\frac{c_{1}}{c}t^{2}+\frac{c_{2}}{c}t^{3}.$$
Now proceed as in example \ref{counter} to show that $5\in Nash_n(\Gamma_n)$.

\section*{Acknowledgements}
We would like to thank Pedro Gonz\'alez P\'erez for explaining us some of his results (joint with B. Teissier) from \cite{GT}. We also 
thank Mark Spivakovsky for stimulating discussions. We are very grateful to a referee for his/her careful reading of the paper, for several 
valuable comments that greatly improved the presentation of some parts of the paper, and for suggesting the study of the independence 
from the generators of a semigroup of the logarithmic Jacobian ideal of higher order.

\vspace{.5cm}
{\footnotesize \textsc {E. Ch\'avez-Mart\'inez, Universidad Nacional Aut\'onoma de M\'exico,} \\
E-mail: enrique.chavez@im.unam.mx}\\
{\footnotesize \textsc {D. Duarte, Universidad Aut\'onoma de Zacatecas-CONACYT.} \\
E-mail: aduarte@uaz.edu.mx}\\
{\footnotesize \textsc {A. Giles Flores, Universidad Aut\'onoma de Aguascalientes.} \\
Email: arturo.giles@cimat.mx}
\end{document}